\documentclass[a4paper]{amsart} 
\usepackage[ascii]{inputenc} 

\usepackage{amssymb}
\usepackage{mathrsfs}
\usepackage[hidelinks]{hyperref}

\usepackage{xcolor}

\newcommand{\cyan}[1]{\textcolor{cyan}{#1}}

\usepackage[shortlabels]{enumitem}
\setlist[enumerate,1]{label={(\Alph*)}}
\setlist[enumerate,2]{label={(\alph*)}}
\setlist[enumerate,3]{label={$\bullet_{\arabic*}$}}

\newenvironment{PROOF}[2][\proofname.]
{\begin{proof}[#1]\renewcommand{\qedsymbol}{\textsquare$_{\rm #2}$}}
{\end{proof}}

\newtheorem{theorem}{Theorem}[section]

\newtheorem{claim}[theorem]{Claim}
\newtheorem{conclusion}[theorem]{Conclusion}

\newtheorem{lemma}[theorem]{Lemma}
\newtheorem{observation}[theorem]{Observation}

\theoremstyle{definition}

\newtheorem{convention}[theorem]{Convention}

\newtheorem{definition}[theorem]{Definition}
\newtheorem{discussion}[theorem]{Discussion}

\newtheorem{thesis}[theorem]{Thesis}

\theoremstyle{remark}

\newtheorem{question}[theorem]{Question}
\newtheorem{remark}[theorem]{Remark}

\newcommand{\Card}{\mathrm{Card}}
\newcommand{\cd}{\mathrm{cd}}

\newcommand{\Sep}{\mathrm{Sep}}

\newcommand{\Sol}{\mathrm{Sol}}

\newcommand{\trp}{\mathrm{trp}}

\newcommand{\BB}{\mathrm{BB}}

\newcommand{\cof}{\mathsf{cof}}

\usepackage{tikz}
\usetikzlibrary{quotes}
\usepackage{float}

\newcommand{\GCH}{\mathsf{GCH}}

\newcommand{\Cohen}{\mathsf{Cohen}}

\newcommand{\cf}{\mathrm{cf}}

\newcommand{\cov}{\mathrm{cov}}

\newcommand{\dom}{\mathrm{dom}}

\newcommand{\inv}{\mathrm{inv}}

\newcommand{\otp}{\mathrm{otp}}

\newcommand{\rang}{\mathrm{rang}}

\newcommand{\bd}{\mathrm{bd}}

\newcommand{\pr}{\mathrm{pr}}

\newcommand{\up}{\mathrm{up}}

\newcommand{\pp}{\mathrm{pp}}

\newcommand{\Reg}{\mathrm{Reg}}

\newcommand{\bfB}{\mathbf{B}}

\newcommand{\bfF}{\mathbf{F}}

\newcommand{\bfT}{\mathbf{T}}
\newcommand{\bfU}{\mathbf{U}}
\newcommand{\bfV}{\mathbf{V}}

\newcommand{\bfp}{\mathbf{p}}

\newcommand{\bbP}{\mathbb{P}}
\newcommand{\bbQ}{\mathbb{Q}}

\newcommand{\bbS}{\mathbb{S}}

\newcommand{\mn}{\medskip\noindent}
\newcommand{\sn}{\smallskip\noindent}
\newcommand{\bn}{\bigskip\noindent}

\newcommand{\cP}{\mathscr{P}}

\newcommand{\clA}{\mathcal{A}}

\newcommand{\clC}{\mathcal{C}}
\newcommand{\clD}{\mathcal{D}}
\newcommand{\clE}{\mathcal{E}}
\newcommand{\clF}{\mathcal{F}}

\newcommand{\clH}{\mathcal{H}}

\newcommand{\clP}{\mathcal{P}}

\newcommand{\clT}{\mathcal{T}}
\newcommand{\clU}{\mathcal{U}}

\newcommand{\gd}{\mathfrak{d}}

\newcommand{\eps}{\varepsilon}
\newcommand{\cl}{c\kern-.11ex \ell}
\newcommand{\lh}{{\ell\kern-.27ex g}}
\newcommand{\rest}{\restriction}

\newcommand{\LL}{\langle}
\newcommand{\RR}{\rangle}
\newcommand{\supminus}{{\text{--}}}

\newcommand{\subref}[1]{$_{\mathrm{\texttt{=}}\mathsf{L{#1}}}$}
\DeclareMathOperator{\divides}{\big|}
\newcommand{\lepref}[1]{({<}\,{#1})}

\newcommand{\olsi}[1]{\,\overline{\!{#1}}} 

\newcommand*{\defeq}{\mathrel{\vcenter{\baselineskip0.5ex \lineskiplimit0pt\hbox{\scriptsize.}\hbox{\scriptsize.}}}=}

\usepackage{tcolorbox}

\newcount\skewfactor
\def\mathunderaccent#1#2 {\let\theaccent#1\skewfactor#2
\mathpalette\putaccentunder}
\def\putaccentunder#1#2{\oalign{$#1#2$\crcr\hidewidth
\vbox to.2ex{\hbox{$#1\skew\skewfactor\theaccent{}$}\vss}\hidewidth}}

\newbox\noforkbox \newdimen\forklinewidth
\setbox0\hbox{$\textstyle\bigcup$}
\setbox1\hbox to \wd0{\hfil\vrule width \forklinewidth depth \dp0
                        height \ht0 \hfil}
\setbox\noforkbox\hbox{\box1\box0\relax}
\def\unionstick{\mathop{\copy\noforkbox}\limits}
\def\nonfork#1#2_#3{#1\unionstick_{\textstyle #3}#2}
\def\nonforkin#1#2_#3^#4{#1\unionstick_{\textstyle #3}^{\textstyle
    #4}#2}
\setbox0\hbox{$\textstyle\bigcup$}
\setbox1\hbox to \wd0{\hfil{\sl /\/}\hfil}
\setbox2\hbox to \wd0{\hfil\vrule height \ht0 depth \dp0 width
                                \forklinewidth\hfil}
\newbox\doesforkbox
\setbox\doesforkbox\hbox{\box1\box0\relax}
\def\nunionstick{\mathop{\copy\doesforkbox}\limits}

\def\fork#1#2_#3{#1\nunionstick_{\textstyle #3}#2}
\def\forkin#1#2_#3^#4{#1\nunionstick_{\textstyle #3}^{\textstyle
    #4}#2}

\newcommand{\stickT}{%
\setbox255=\hbox{\raise1ex\hbox{$\hspace{0.2pt}\,\bullet\,$}}
\mathord{\rlap{\hbox to\wd255{\hss\hbox{$|$}\hss}}
\box255}
}
\newcommand{\stickS}{%
\setbox255=\hbox{\raise0.6ex\hbox{$\scriptstyle\bullet$}}
\mathord{\rlap{\hbox to\wd255{\hss\hbox{$\scriptstyle|$}\hss}}
\box255}
}

\author[S. Shelah]{Saharon Shelah}
\address{Einstein Institute of Mathematics,
The Hebrew University of Jerusalem,
9190401, Jerusalem, Israel; and\\
Department of Mathematics,
Rutgers University,
Piscataway, NJ 08854-8019, USA}
\urladdr{https://shelah.logic.at/}
\thanks{First typed 2025-05-29.
The author thanks Craig Falls  
for generously funding typing services, and Matt Grimes for the careful and beautiful typing.
The author would like to thank the Israel Science Foundation for partial support of
this research by grant 2320/23 (2023-2027).\\
References like (e.g.) [Sh:950, Th.0.2\subref{y5}] mean that \textsf{y5} is the internal label of Theorem 0.2 in the TeXfile of [Sh:950].
The reader should note that the version in my website is usually more up-to-date than the one in arXiv.
This is publication number    
1268
in Saharon Shelah's list.
}

\makeatletter
\@namedef{subjclassname@2020}{\textup{2020} Mathematics Subject Classification}
\makeatother
\date{January 16, 2026} 

\title{Super Black Boxes Revisited\\{} 1268}

\begin{document}
\makeatletter\def\shfiuwefootnote{\gdef\@thefnmark{}\@footnotetext}\makeatother\shfiuwefootnote{Version 2026-01-18. See \url{https://shelah.logic.at/papers/1268/} for possible updates.}
\begin{abstract} 
    Let $\kappa,\theta < \lambda$ be cardinals, with $\lambda$ and $\kappa$ regular. 
    Concentrating on a simple case, we say that the triple $(\lambda,\kappa,\theta)$ has a \emph{Super Black Box} \underline{when} the following holds.

    For some stationary $S \subseteq \{\delta < \lambda : \cf(\delta) = \kappa\}$ and $\olsi C = \LL C_\delta : \delta \in S \RR$,  where $C_\delta$ is a club of $\delta$ of order type $\kappa$, for every coloring
    $\olsi F = \LL F_\delta : \delta \in S \RR$ with 
    $F_\delta : {}^{C_\delta}\!\lambda \to \theta$, there exists 
    $\LL c_\delta : \delta \in S\RR \in {}^S\!\theta$ such that for every 
    $f : \lambda \to \theta$, for stationarily many $\delta \in S$, we have $F_\delta(f \rest C_\delta) = c_\delta$.
    
    In an earlier work, it was proved (along with much more) that for a class of cardinals $\lambda$ this holds for many pairs $(\kappa,\theta)$. (E.g.~$\kappa < \aleph_\omega$ is large enough and $\beth_\omega(\theta) < \lambda$.) 
    However, the most interesting cases (at least with regards to Abelian groups) are $\kappa = \aleph_0,\aleph_1$, which have not been covered there. 

    Here we restrict ourselves to the case where $\olsi F$ is a {so-called} \emph{continuous coloring}, 
    which includes the case where $F_\delta$ is computed from some 
    $$
    \big\LL F_{\delta,\beta}'(f \rest (C_\delta \cap \beta)) : \beta \in C_\delta \big\RR.
    $$
    This covers the cases we have in mind.
    We mainly prove results without any other caveats: e.g.
    \begin{itemize}
        \item For every $\theta$ and regular $\kappa$ there exists such a $\lambda$.
    \end{itemize}

    We also deal with having multiple {$\olsi C$-s}, and the existence of quite free subsets of ${}^\kappa\mu$.
\end{abstract}
\maketitle

\bn 
\centerline{\textbf{\underline{Annotated Content}}}

\bn
\textbf{\S0. Introduction}  \hfill p. \pageref{S0}

\bn
\textbf{\S1. The Framework} $_{(\mathsf{label\ x})}$ \hfill p. \pageref{S1}

\bn
\textbf{\S2. The Black Box Property} $_{(\mathsf{label\ a})}$ \hfill p.\pageref{S2}
\begin{quotation}
    (We prove it for $\lambda \defeq \cf(2^\mu)$.)
\end{quotation}

\bn
\textbf{\S3. The DBB Property} $_{(\mathsf{label\ b})}$ \hfill p.\pageref{S3}
\begin{quotation}
    Here we allow many $C_\delta \subseteq \delta$, but require only $\mu < \lambda \leq 2^\mu < 2^\lambda$.
\end{quotation}

\bn
\textbf{\S4. When do quite free subsets $\Lambda \subseteq {}^\kappa\!\mu$ exist?} \hfill p.\pageref{S4}
\begin{quotation}
    E.g.\ if $\mu$ is strong limit singular and $\lambda \in (\mu,2^\mu)$, \underline{then} there is a $\mu^+$-free set 
    $\Lambda \subseteq {}^{\cf(\mu)}\mu$ of cardinality $\lambda$. Earlier, this result was known for \emph{almost} all such $\mu$-s.
\end{quotation}


\newpage
\setcounter{section}{-1}
\section{Introduction}\label{S0}

We continue (but do not rely on) \cite{Sh:775} and \cite{Sh:898},\footnote{
    Well, except for quoting one result of \cite{Sh:898} in \S3.
} 
while \cite{Sh:1028} presents another direction we could pursue. Compared to \cite{Sh:775}, we restrict ourselves to the so-called \emph{continuous colorings}, \underline{but} the conditions on $\kappa$ are greatly weakened. 

Recalling the BB Trichotomy Theorem from \cite[1.22\subref{h.7}]{Sh:898}, Case (B) there will be expanded upon in \S2 here, and \S3 will examine cases (C) and (A). 

For the Trivial Dual Conjecture on abelian groups, see \cite{Sh:1028}, \cite{Sh:883}, and \cite{Sh:898}.

\mn
We believe:
\begin{thesis}
Proving theorems with assumptions on cardinal arithmetic is better than just giving consistency results (usually via forcing). Another candidate for such hypotheses in $\bfV$ is of an inner model close to it.

The most famous cardinal-arithmetic assumption is the $\GCH$, as it resolves many questions and makes many theorems easy to prove. But we believe that assuming some failures of $\GCH$ at specific cardinals can be illuminating as well.
\end{thesis}

\sn
Results related to this have applications for constructing Abelian groups and modules. Hopefully we will be able to apply the present results in \cite{Sh:1045} and \cite{Sh:F2448}.

\begin{definition}\label{yfree}
1) Let $\Lambda$ be a set of sequences of length $\kappa$. We say that $\Lambda$ is \emph{free} \underline{when} there exists a function $f : \Lambda \to \kappa$ such that 
$$
\big\LL\{\eta(i) : f(\eta) \leq i < \kappa\} : \eta \in \Lambda\big\RR
$$
is a sequence of pairwise disjoint sets.

(This definition is easily adaptable to (e.g.) $\Lambda \subseteq [\mu]^\kappa$.)

\mn
2) We say $\Lambda$ is $\mu$-\emph{free} if every subset of $\Lambda$ of cardinality $< \mu$ is {free}.
\end{definition}

\medskip
We would like to construct (e.g.) suitable $\lambda$-free Abelian groups. For this we may use the following fact:
\begin{enumerate}
    \item [$\boxplus_1$] Suppose $\mu$ is strong limit singular and
    $$
    \kappa \defeq \cf(\mu) < \mu < \lambda = \cf(\lambda) < 2^\mu < 2^\lambda.
    $$
    \underline{Then} there exists a $\mu$-free subset of ${}^\kappa\mu$ of cardinality $\lambda$.
\end{enumerate}
[Why? If $\cf(\mu) > \aleph_0$ this is proved in \cite[Ch.II, \S3]{Sh:g}. Note that the proof there has been extended to many cardinals with $\cf(\mu) = \aleph_0$. In \S4 we shall prove this for \emph{all} strong limit cardinals of cofinality $\geq\aleph_0$.]

\sn
Also recall (e.g., from \cite{Sh:898}):
\begin{enumerate}
    \item [$\boxplus_2$] Suppose $\mu$ and $\kappa$ are as above, and there is $\chi < \chi^{<\kappa} = 2^\mu$.
    \underline{Then} there exists a $\mu$-free subset of ${}^\kappa\mu$ of cardinality $2^\mu$.
\end{enumerate}

We have reasonable Black Boxes (see \cite{Sh:775}, \cite{Sh:898}) {and more here}.
The proofs in \cite{Sh:775} cover many specific cases (e.g.\ the result mentioned in the abstract). Say, if $\lambda \defeq \cf(2^\mu) > \beth_\omega$, then for every large enough regular $\kappa < \aleph_\omega$ we have a black box on some $\olsi C = \LL C_\delta : \delta \in S_\kappa^\lambda\RR$. We shall prove it here for \emph{all} such $\kappa$. (See \ref{y2} below.)

\mn
\begin{convention}\label{y0}
0) For $\varnothing \in D \subseteq \clP(\lambda)$, let $D^+ \defeq \{A \subseteq \lambda : \lambda \setminus A \notin D\}$.

Note that if $J$ is an ideal on $\lambda$, then $J^+ = \clP(\lambda) \setminus J$.

\mn
1) Let $\up(\lambda)$ be the set of non-empty upward closed $\clD \subset \clP(\lambda)$. 

E.g.\ a filter on $\lambda$ is an example of such a set.

\mn
2) $\cof(\kappa)$ is the class of ordinals $\{\delta : \cf(\delta) = \cf(\kappa)\}$ and $\cof\lepref{\kappa} \defeq \bigcup\limits_{\theta<\kappa} \cof(\theta)$. (Usually $\kappa$ is a regular cardinal.)

\mn
3) For $\kappa$ regular, $S_\kappa^\lambda \defeq \lambda \cap \cof(\kappa) = \{\delta < \lambda : \cf(\delta) = \kappa\}$ and $S_{<\kappa}^\lambda \defeq \lambda \cap \cof\lepref{\kappa}$.
\end{convention}

\bigskip
\centerline{*\qquad*\qquad*}

\mn
We shall aim to state our results concisely, rather than with maximum possible generality. A major one is the following.
\begin{theorem}\label{y2}
`If $(A)$ then $(B)$,' where:
\begin{enumerate}[$(A)$]
    \item 
    \begin{enumerate}[$(a)$]
        \item $\mu \geq \theta = \theta^{2^\kappa}$, $\kappa$ is regular, and $\alpha < \mu \Rightarrow |\alpha|^\kappa \leq \mu$.
\sn
        \item $\lambda \defeq \cf(2^\mu)$ and $S \subseteq S_\kappa^\lambda$ is stationary.
\sn
        \item For each $\delta \in S$, we have $C_\delta' \subseteq \delta = \sup(C_\delta')$ with $\otp(C_\delta') = \kappa$.
\sn
        \item For all $\beta < \lambda$ we have $C_\beta^\bullet \subseteq \beta$ such that $2^{|C_\beta^\bullet|} \leq 2^\mu$.
    \end{enumerate}
\sn
    \item There exists $\olsi C = \LL C_\delta : \delta \in S\RR$ with $C_\delta \subseteq C_\delta'$ and $\sup(C_\delta) = \delta$ such that \underline{if} $F_\beta : {}^{C_\beta^\bullet}(2^\mu) \to \theta$ 
    for $\beta < \lambda$
    \underline{then}
    there exists $\LL c_{\delta,\beta} : \delta \in S,\ \beta \in C_\delta'\RR$ with $c_{\delta,\beta} < \theta$ such that for every 
    $\eta \in {}^\lambda(2^\mu)$, for stationarily many $\delta \in S$, we have
    $$
    \beta \in C_\delta \Rightarrow F_\beta(\eta \rest C_\beta^\bullet) = c_{\delta,\beta}.
    $$
\end{enumerate}
\end{theorem}

\mn 
The proof can be found on page \pageref{a43}.

\sn
\begin{theorem}\label{y5}
$1)$ In \emph{\ref{y2}}$(A)(a)$, we can weaken the demand 
$(\forall \alpha < \mu)\big[|\alpha|^\kappa \leq \mu \big]$ to
$``\bfU_\kappa(\mu) = \mu > \theta."$ (See Definition \emph{\ref{x7}(1)}.)

\mn
$2)$ We can replace clause \emph{\ref{y2}(A)(d)} via the use of $\bfp$ as in \emph{\ref{a2}(1)}.
\end{theorem}

\mn
An additional result is as follows. (See \S3-4 for context --- specifically,  \ref{b2} and \ref{b5}.)

\begin{theorem}\label{y8}
Assume $\mu$ is strong limit singular, $\kappa \defeq \cf(\mu)$, 
$$
\kappa+\theta < \mu < \lambda = \cf(\lambda) < 2^\mu < 2^\lambda,
$$ 
and $S$ is a stationary subset of $S_\kappa^\lambda$.

\underline{Then} we can find $\olsi C = \LL C_\gamma^\delta : \delta \in S,\ \gamma < \lambda\RR$ such that:
\begin{enumerate}[$\bullet_1$]
    \item $C_\gamma^\delta \subseteq \delta = \sup(C_\gamma^\delta)$
\sn
    \item $\otp(C_\gamma^\delta) = \kappa$
\sn
    \item $\olsi C$ is a $\mu^+$-\emph{free sequence}.
        
    By this we mean: if $u \subseteq S \times \lambda$ is of cardinality 
    $< \mu^+$, \underline{then} there exists some sequence 
    $\LL \beta_\gamma^\delta : (\delta,\gamma) \in u\RR$ with
    $\beta_\gamma^\delta \in C_\gamma^\delta$ such that
$$
    \LL C_\gamma^\delta \setminus \beta_\gamma^\delta : (\delta,\gamma) \in u\RR
$$
    is a sequence of pairwise disjoint sets.
\sn
    \item  If $\bfF_\gamma^\delta : {}^{C_\gamma^\delta}(2^\mu) \to \theta$ 
    for $(\delta,\gamma) \in S \times \lambda$, \underline{then} we can find a 
    $$
    \bar c^\delta = \LL c_\gamma^\delta : \gamma < \lambda\RR 
    \in {}^\lambda\theta
    $$ 
    such that for any $\delta \in S$ and 
    $f : \delta \to 2^\mu$, for some (even `many') $\gamma < \lambda$, we have 
    $$
    \bfF_\gamma^\delta( f \rest C_\gamma^\delta) = c_\gamma^\delta.
    $$
\end{enumerate}
\end{theorem}

\bn
We may rephrase Theorem \ref{y2} as follows:
\begin{theorem}\label{y11}
Suppose the assumptions in \emph{\ref{y2}(A)} all hold. 
\underline{Then} for some $\olsi C = \LL C_\delta : \delta \in S\RR$ with $C_\delta \subseteq \delta = \sup(C_\delta)$ and $C_\delta \subseteq \cyan{C_\delta'}$,
we have 
$$
\BB_*(\lambda,\olsi C,\olsi C^\bullet,2^\mu,\theta,\kappa).
$$

By this we mean ($\kappa$ is regular, and) clauses $(A)(b)$-$(d)$ and $(B)$ of \emph{\ref{y2}} all hold. I.e.
\begin{enumerate}
    \item [$\boxplus$]
    \begin{enumerate}[$(a)$]
        \item $\kappa$ is regular.
\sn
        \item $\lambda \defeq \cf(2^\mu)$ and $S \subseteq S_\kappa^\lambda$ is stationary.
\sn
        \item For each $\delta \in S$, we have $C_\delta \subseteq \delta = \sup(C_\delta)$ with $\otp(C_\delta) = \kappa$.
\sn
        \item For all $\beta < \lambda$, we have $C_\beta^\bullet \subseteq \beta$ such that $2^{|C_\beta^\bullet|} \leq 2^\mu$.

        (Without loss of generality $\beta \in \lambda \setminus \bigcup\limits_\delta {C_\delta} \Rightarrow C_\beta^\bullet = \varnothing$.)
\sn
        \item {If} $F_\beta : {}^{C_\beta^\bullet}(2^\mu) \to \theta$ \underline{then} there exists $\LL c_{\delta,\beta} : \delta \in S,\ \beta \in C_\delta'\RR$ with $c_{\delta,\beta} < \theta$ such that for every $\eta \in {}^\lambda(2^\mu)$, for stationarily many 
        $\delta \in S$, we have
$$
        \beta \in C_\delta \Rightarrow F_\beta(\eta \rest C_\beta^\bullet) = c_{\delta,\beta}.
$$
    \end{enumerate}
\end{enumerate}
\end{theorem}

\bn
We also prove an analogous result for Double Black Boxes.
\begin{theorem}\label{y14}
We have $`(A) \wedge (B) \Rightarrow (C)$,\emph{'} where
\begin{enumerate}[$(A)$]
    \item 
    \begin{enumerate}[$(a)$]
        \item $\theta + \kappa < \mu < \min\{\lambda,\lambda_*\} \leq \lambda_* + \lambda \leq 2^\mu$
\sn
        \item $\kappa \in \Reg \cap \lambda$
\sn
        \item $S \subseteq S_\kappa^\lambda$ is stationary, and $\delta \in S \Rightarrow \mu^2 \divides \delta$.
\sn
        \item $\olsi C = \LL C_\delta : \delta \in S\RR$ with 
        $C_\delta \subseteq \delta = \sup(C_\delta)$, 
        $\otp(C_\delta) = \kappa$, and
        $$
        \alpha \in C_\delta \Rightarrow \mu \divides \alpha.
        $$
        \item $D_\delta$ is a filter on $\lambda_*$. (The default is the club filter
        if $\lambda_*$ is regular, $\{A \subseteq \lambda_* : |\lambda_* \setminus A| < \lambda_*\}$ if it is singular, and $\{\lambda_*\}$ if it is finite.)
    \end{enumerate}
\sn
    \item 
    \begin{enumerate}[$(a)$]
        \item $\mu$ is strong limit singular, and $\mu_* \leq \mu^+$.
\sn
        \item $\lambda_* \defeq \min\{\partial : 2^\partial > 2^\mu\} \leq \cf(\lambda) \leq 2^\mu$
\sn
        \item $\lambda < 2^\mu \wedge \mu_* \defeq \mu^+$, \underline{or} there is a $\mu_*$-free subset of ${}^\kappa\!\mu$ of cardinality $\lambda$.
    \end{enumerate}
\sn
    \item We can find $\olsi C^\bullet$ such that 
    $$
    \mathrm{DBB}_*(\lambda,\olsi C,\olsi C^\bullet,2^\mu,\theta,\kappa)
    $$
    holds. This means
    \begin{enumerate}[$(a)$]
        \item Clauses $(A)(a)$-$(d)$ are satisfied.
\sn
        \item $\olsi C^\bullet = \LL C_\gamma^\delta : \delta \in S,\ \gamma < \lambda_*\RR$, where $C_\gamma^\delta \subseteq \bigcup\limits_{\beta \in C_\delta} [\beta,\beta+\mu)$ and 
        $$
        \beta \in C_\delta \Rightarrow \big| C_\gamma^\delta \cap [\beta,\beta+\mu) \big| = 1.
        $$
        \item $\olsi C^\bullet$ is $\mu_*$-\emph{free} (as defined in \emph{\ref{y8}$\bullet_3$}).\footnote{
            I.e.\ for all $u \subseteq S \times \lambda_*$ of cardinality $< \mu_*$ there exist $\LL \beta_\gamma^\delta : (\delta,\gamma) \in u\RR$ such that 
$$
            \big\LL C_\gamma^\delta \setminus \beta_\gamma^\delta : (\delta,\gamma) \in u \big\RR
$$
            is a sequence of pairwise disjoint sets.
        }
\sn
        \item If $F_\gamma^\delta : {}^{ C_\gamma^\delta}(2^\mu) \to \theta$ (for $(\delta,\gamma) \in S \times \lambda_*$) \underline{then} for some\\ $\LL c_\gamma^\delta : \gamma < \lambda_*,\ \delta \in S\RR$ with $c_\gamma^\delta < \theta$, for every $f : \lambda \to 2^\mu$ and stationarily many $\delta \in S$, some\footnote{
            As in \ref{y8}, it may be possible to strengthen this to 
            `many $\gamma < \lambda$' for some definition of \emph{many}. 
        } 
        $\gamma < \lambda_*$, we have
        $$
        F_\gamma^\delta(f \rest C_\gamma^\delta) = c_\gamma^\delta.
        $$
    \end{enumerate}
\end{enumerate}
\end{theorem}

\bn
\begin{claim}\label{y17}
$1)$ If \emph{\ref{y14}(A)} holds and $\lambda_* \defeq 1$, then \emph{DBB}$_*$ is equivalent to $\BB_*$.

\mn
$2)$ In \emph{\ref{y14}} (as in \emph{\ref{y8}}) we may change clause $(C)(d)$ to use $\bfF_\gamma^\delta : {}^{C_\delta^\bullet}(2^\mu) \to \theta$, where $C_\delta^\bullet \defeq \bigcup\limits_{\beta \in C_\delta} [\beta,\beta+\mu)$.
\end{claim}

\bn
\begin{definition}\label{y20}  
1) For a regular uncountable cardinal $\lambda$, let 
$$
\check I[\lambda] \defeq \big\{S \subseteq \lambda : 
\text{some pair $(E,\bar a)$  satisfies part (2) below}\big\}.
$$ 

\sn
2) We say that $(E,\bar a)$ is a witness for $S \in \check I[\lambda]$
\underline{when} ($ S \subseteq \lambda $ and): 
\begin{enumerate}
    \item  $E$ is a club of the regular cardinal $\lambda$.
\sn
    \item  $\bar a = \LL a_\alpha : \alpha < \lambda \RR$, $a_\alpha \subseteq \alpha$, and $\beta \in a_\alpha \Rightarrow a_\beta = \beta \cap a_\alpha$.
\sn
    \item  For every $\delta \in E \cap S$, $a_\delta$ is an unbounded subset 
    of $\delta$ of order-type $< \delta$.
\end{enumerate}

\sn
3) For $\kappa < \lambda$, we define $\check I_\kappa[\lambda] \defeq \big\{S \subseteq \lambda : S \cap S_{\cf(\kappa)}^\lambda \in \check I[\lambda] \big\}$.
\end{definition}

\bn
By \cite{Sh:108}, \cite{Sh:420} and \cite{Sh:E12}:
\begin{claim}\label{y23}  
Let $\lambda$ be regular uncountable. 

\sn
$1)$  We have 
$S \in \check I[\lambda]$ 
\underline{iff}  
we can find a witness $(E,\bar a)$ for it which satisfies:
\begin{enumerate}
    \item[{$(a)$}]  $\delta \in S \cap E \Rightarrow \otp(a_\delta) = \cf(\delta)$
\sn
    \item[{$(b)$}]  If $\alpha \notin S$ then $\otp(a_\alpha) < \cf(\delta)$ for some $\delta \in S \cap E$.
\end{enumerate}

\sn
$2)$ $S \in \check I[\lambda]$ \underline{iff} there is a pair $(E,\olsi \cP)$ which is a \emph{weak witness} for it.
By this we mean:
\begin{enumerate}
    \item[{$(a)$}]  $E$ is a club of the regular uncountable $\lambda$.
\sn
    \item[{$(b)$}]  $\olsi \cP = \LL \cP_\alpha : \alpha < \lambda\RR$, 
    where $\cP_\alpha \subseteq \clP(\alpha)$ has cardinality
    $< \lambda$.
\sn
    \item[{$(c)$}]  If $\alpha < \beta < \lambda$ and 
    $\alpha \in u \in \cP_\beta$ \underline{then} $u \cap \alpha \in \cP_\alpha$.
\sn
    \item[{$(d)$}]  If $\delta \in E \cap S$ \underline{then} some 
    $u \in \cP_\delta$ 
    is an unbounded subset of $\delta$ of order type $\le \cf(\delta)$.
    (We may restrict ourselves to the case where $\delta$ is a limit ordinal.)
\end{enumerate}

\sn
$3)$ Suppose $(S,E,\olsi\cP)$ are as in part $(2)$ and $C$ is another club of $\lambda$. \underline{Then} the triple $(S_*,E_*,\olsi\cP_{\!*})$, defined below, satisfies part $(2)$ as well.
\begin{itemize}
    \item $S_* \defeq S \cap C$
\sn
    \item $E_* \defeq \{\delta \in S : \delta = \sup(E \cap C \cap \delta)\}$
\sn
    \item $\olsi\cP_{\!*} = \LL \cP_{\!\alpha}^* : \alpha < \lambda\RR$, where $\cP_{\!\alpha}^* \defeq 
    {\{}\{\sup(E \cap C \cap \beta) : 
    {\beta \in C \}:
    C \in \cP_{\!\alpha}}\}$
\end{itemize}

\sn
$4)$ If $\lambda$ is regular then $S_{<\lambda}^{\lambda^+} \in \check I[\lambda]$.
\end{claim}

\newpage
\section{The framework}\label{S1}

We will open by quoting some definitions from \cite{Sh:775} (although that paper is not a prerequisite). We investigate the notion of $\Sep(-)$ and define some relatives which we will need.

In \S2 we will use only $\Sep_3$ (although $\Sep_2$ would actually be sufficient for proving \ref{y2}, \ref{y5}). 

\begin{convention}\label{x1}
What we call $\BB$ here will be denoted as $\BB^0$ in later sections.
\end{convention}

\mn
\begin{definition}\label{x5}
Assume $\lambda > \kappa$ are regular cardinals, and let 
$\chi \leq \lambda$. Let 
$$
S \subseteq \{\delta<\lambda : \cf(\delta) = \kappa\}
$$ 
be a stationary subset of $\lambda$.

\mn 
1) We say $\bfp = \big\LL (C_\delta,C_\delta') : \delta \in S\big\RR$ is a \emph{$(\lambda,\kappa,\chi)$-$\BB$-parameter} \underline{when}:
\begin{enumerate}
    \item $C_\delta \subseteq \delta$, with $\sup(C_\delta) = \delta$ and $|C_\delta| < \chi$, \underline{or} just $\otp(C_\delta) \leq \chi$.
\sn
    \item $C_\delta' \subseteq \delta$, $\sup(C_\delta') = \delta$, and $\otp(C_\delta') = \kappa$. (We do not require that $C_\delta$ \emph{or} $C_\delta'$ be closed in $\delta$.)
\sn
    \item For all $\alpha < \lambda$, the set 
    $$
    \big\{ C_\delta \cap \alpha : \delta \in S,\ C_\delta' \ni \alpha \big\}
    $$
    has cardinality $<\lambda$.
\end{enumerate}

\mn
1A) We say {$\bfp$} is \emph{good} \underline{when} in addition, 
\begin{enumerate}
    \item [(C)$^+$] For all $\alpha < \lambda$ the set\footnote{
        $C_\delta \cap \alpha$ will suffice.
    } 
    $\big\{ (C_\delta \cap \alpha,C_\delta' \cap \alpha) :  \delta \in S,\ C_\delta' \ni \alpha\big\}$ has cardinality $<\lambda$.
\end{enumerate}

\mn
1B) If $\bfp$ just satisfies (1)(A)-(B), we call it a \emph{weak} $(\lambda,\kappa,\chi)$-$\BB$-parameter.

\mn
2A)
We say that $\bfp$ \emph{does $D$-guess clubs}, where $D$ is a filter on $\lambda$, \underline{when} for every club $E \subseteq \lambda$, 
$$
\{\delta \in S : C_\delta' \subseteq E\} \in D^+.
$$

\mn
2B) For $\bfp$ as above,
\begin{enumerate}[(a)]
    \item $\bar\beta_\delta = \LL \beta_{\delta,i} : i < \kappa\RR$ will list the elements of $C_\delta'$ in increasing order.
\sn
    \item $\beta_{\delta,<i} = \beta(\delta,<i) \defeq \bigcup\limits_{j<i} (\beta_{\delta,j}+1)$.
\end{enumerate}

\mn
2C) We may write $\lambda_\bfp$, $\kappa_\bfp$, $\beta_{\delta,i}^\bfp$, etc.~whenever there are multiple BB-parameters under discussion, or the identity of $\bfp$ is otherwise unclear from context.
 
\mn
2D) If $(\forall \delta \in S)[ C_\delta = C_\delta']$, then we may write $\bfp = \LL C_\delta : \delta \in S\RR$. We may omit $\chi$ when $\chi \defeq \min\{\theta : \delta \in S \Rightarrow \otp(C_\delta) \leq \theta\}$.

\mn
3) We say that $\olsi F = \LL F_\delta : \delta \in S\RR$ is a $(\bfp,\partial,\theta)$-\emph{coloring} if $\theta \geq 2$, $\partial \geq 2$, and $F_\delta : {}^{C_\delta}\partial \to \theta$.

\mn
4) Let $\olsi F$ (and $\bfp,\partial,\theta$) be as above, and $D$ be a filter on $\lambda$. (The default choice will be the club filter.)

We say $\bar c \in {}^S\!\theta$ (or $\in {}^\lambda\theta$) is a 
$\bfp$-$D$-$\olsi F$-\emph{BB-sequence} \underline{if} for every $\eta \in {}^\lambda\partial$ the set $\{\delta \in S : F_\delta(\eta \rest C_\delta)  = c_\delta\}$ is a member of $D^+$ (and in the default case, a stationary subset of $\lambda$).

\mn
5) We may omit $\bfp$ if both $\olsi C$ and $\olsi C'$ are clear from the context.

\mn
6) We say $\olsi C = \LL C_\delta : \delta \in S\RR$ is 
$(\lambda,\kappa)$-\emph{good} \underline{when} 
\begin{enumerate}
    \item $S$ is a stationary subset of $\{\delta < \lambda : \cf(\delta) = \kappa\}$ and a member of $\check I_\kappa[\lambda]$.
\sn
    \item $C_\delta \subseteq \delta = \sup (C_\delta)$
\sn
    \item $\otp(C_\delta) = \kappa$
\sn
    \item For every $\beta < \lambda$ the set $\{C_\delta \cap \beta : \beta \in C_\delta,\ \delta \in S\}$ has cardinality $<\lambda$.
\end{enumerate}
\end{definition}

\mn
\begin{claim}\label{x6}
Assume $\lambda > \kappa$ are regular cardinals and $\chi \in [\kappa,\lambda]$.

\sn
$1)$ If $S$ is a stationary subset of $S_\kappa^\lambda \defeq \{\delta < \lambda : \cf(\delta) = \kappa\}$ \underline{then} there exists a weak $(\lambda,\kappa,\chi)$-$\BB$-parameter $\bfp$ with $S_\bfp = S$.

\sn
\emph{1A)} If $\chi = \lambda$ \underline{then} we may omit ``weak," and set $C_{\bfp,\delta} \defeq \delta$ for all $\delta \in S$.

\sn
\emph{1B)} If $\lambda \defeq \kappa^+$ and $S \subseteq S_{<\kappa}^\lambda$ is stationary, \underline{then} we can add \emph{``}for some club $E \subseteq \lambda$, there exists a a good $\bfp$ with $\olsi C_\bfp =  \LL C_\delta : \delta \in S \cap E\RR$ such that each $C_\delta$ is a club of $\delta."$

\sn
$2)$ If $S$ is a stationary subset of $S_\kappa^\lambda$ and a member of 
$\check I_\kappa[\lambda]$, \underline{then} in part $(1)$ we may also add $``\bfp$ is \emph{good}, with $S_\bfp \defeq S \cap E$ for some 
club $E$ of $\lambda$."

\sn
$3)$ If $\lambda > \kappa^+$ \underline{then} there exists a good $(\lambda,\kappa,\chi)$-$\BB_*$-parameter.

\sn
$4)$ Every good $(\lambda,\kappa,\chi)$-$\BB$-parameter is a $(\lambda,\kappa,\chi)$-$\BB$-parameter, and every $(\lambda,\kappa,\chi)$-$\BB$-parameter is a weak one.
\end{claim}

\begin{PROOF}{\ref{x6}}
Easy. 

E.g.\ for part (3), use \cite[\S1]{Sh:420}.
Part (1B) follows by \cite[4.4]{Sh:351}.
\end{PROOF}

\mn
\begin{definition}\label{x7}
Let $\kappa \leq \mu$.

\sn
1) We define $\bfU_{\!\kappa}(\mu)$ to be 
$$
\min\!\big\{ |\clU| : \clU \subseteq [\mu]^\kappa,\ (\forall v \in [\mu]^\kappa)(\exists u \in \clU) \big[ |u \cap v| = \kappa \big] \big\}.
$$

\sn
2) Let $\bfU_{\!\kappa}'(\mu)$ mean
$$
\min\!\big\{ |\clF| : \clF \subseteq {}^\kappa\!\mu \text{ and } (\forall g \in {}^\kappa\!\mu)(\exists f \in \clF)(\exists^\kappa i < \kappa ) \big[f(i) = g(i)\big]  \big\}.
$$

\mn
3) If $\clD \in \up(\kappa)$ then we let
$$
\bfU_\clD(\mu) \defeq \min\!\big\{ |\clF| : \clF \subseteq {}^\kappa\!\mu  \text{ and } (\forall g \in {}^\kappa\!\mu)(\exists f \in \clF) \big[ \big\{i < \kappa : f(i) = g(i) \big\} \in \clD\big]  \big\}.
$$

\mn
4) Accordingly, if $J$ is an ideal on $\kappa$ (or an ideal on some set $X$ containing $[X]^{\aleph_0}$) then we may write $\bfU_{\!J}(\mu)$ as shorthand for $\bfU_{\!J^+}(\mu)$, to keep notation consistent with \cite{Sh:775} and others.

\end{definition}

\mn
Obviously,
\begin{observation}\label{x10}
$1)$ If $\mu \geq 2^\kappa$ then $\bfU_\kappa(\mu) = \bfU_\kappa'(\mu)$.

\mn
$2)$ If $\mu = \mu^\kappa$ (or just $\alpha < \mu \Rightarrow |\alpha|^\kappa \leq \mu$ and $\cf(\mu) \neq \kappa$), then $\bfU_\kappa(\mu) = \bfU_\kappa'(\mu) = \mu$.
\end{observation}

\mn
\begin{definition}\label{x8} \ 

\noindent
1) When we write $\Sep(\chi,\mu,\partial,\theta,\Upsilon)$, we mean that there exists $\bar f = \LL f_\eps : \eps < \chi\RR$ such that:
\begin{enumerate}
    \item $f_\eps : {}^\mu\partial \to \theta$
\sn
    \item For every $\varrho \in {}^\chi\theta$, the set\footnote{
        `Sol' stands for \emph{solution}.
    }
    $\Sol_\varrho \defeq \big\{ \nu \in {}^\mu\partial : (\forall \eps < \chi) \big[ f_\eps(\nu) \neq \varrho(\eps) \big] \big\}$ has cardinality $< \Upsilon$.

    (The reader may assume $\Upsilon \leq \partial^\mu$, as the condition is vacuous otherwise.)
\end{enumerate}
Such a sequence $\bar f$ will be called a \emph{witness} for 
$\Sep(\chi,\mu,\partial,\theta,\Upsilon)$.

\mn
1A) If $\partial \defeq \theta$, we may omit it. 

\mn
2) We write $\Sep(\mu,\partial,\theta)$ to mean that $\Sep(\mu,\mu,\partial,\theta,\Upsilon)$ holds for some regular $\Upsilon \leq 2^\mu$. 

As in part (1A), $\Sep(\mu,\theta) \defeq \Sep(\mu,\theta,\theta)$.

\mn
2A) $\Sep({<}\,\mu,\theta)$ will mean that $\Sep(\sigma,\mu,\theta,\theta,\Upsilon)$ holds for some $\sigma < \mu$ and $\Upsilon$ as in part (2).


\mn
3) We may write $\Sep_1$ instead of $\Sep$, to distinguish it from $\Sep_2$ in \ref{x17} and $\Sep_3$ in \ref{x38}.
\end{definition}

\medskip
In \cite[1.11\subref{d.7}]{Sh:775} we showed $\Sep(\mu,\theta)$ holds for many values of $\mu$ and $\theta$. The following is a generalization of that theorem.
\begin{claim}\label{x20}
If at least one of the following holds, \underline{then} we have $\Sep(\mu,\mu,2^\mu,\theta,\Upsilon)$:
\begin{enumerate}[$(a)$]
    \item $\mu = \mu^\theta$ and $\Upsilon \defeq \theta$.
\sn
    \item $\bfU_\theta(\mu) = \mu$ and $\Upsilon \defeq (2^{<\theta})^+ \leq 2^\mu$.
\sn
    \item We have $\bfU_{[\sigma]^{<\theta}}(\mu) = \mu$ and $\Upsilon \defeq (2^{<\sigma})^+ \leq 2^\mu$ for some $\sigma \geq \theta$ such that $\sigma^\theta \leq \mu$.
\sn
    \item $\theta = \cf(\theta) < \mu$, $\mu$ is strong limit singular of cofinality $\neq \theta$, and (e.g.) $\Upsilon \defeq (2^{\theta+\cf(\mu)})^+$.
\sn
    \item $\Upsilon \defeq \beth_\omega(\theta) \leq \mu$.
\end{enumerate}
\end{claim}

\begin{PROOF}{\ref{x20}}
Let 
$\bar\eta = \LL \eta_\beta : \beta < 2^\mu\RR$ list ${}^\mu(2^\mu)$ 
without repetition. 
Let $\chi \defeq \big(2^{2^\mu}\big)^+$, and let $N \prec (\clH(\chi),\in)$ be of cardinality $\mu$ such that $(\mu+1) \cup \{\bar\eta\} \subseteq N$. Let 
$\bar f = \LL f_\eps : \eps < \mu\RR$ list all the functions from ${}^\mu(2^\mu)$ 
to $\theta$ which are members of $N$. It will suffice to prove
\begin{enumerate}
    \item [$\boxplus_1$]  For every  
    $\varrho \in {}^\mu(2^\mu)$, the set
    $$
    \Sol_\varrho \defeq \big\{ \eta \in {}^\mu(2^\mu) : (\forall \eps < \mu) \big[ f_\eps(\eta) \neq \varrho(\eps) \big] \big\}
    $$
    has cardinality $< \Upsilon$.
\end{enumerate}
For \emph{this} it will suffice to prove
\begin{enumerate}
    \item [$\boxplus_2$] For every $\Lambda \subseteq {}^\mu(2^\mu)$ of cardinality $\Upsilon$, for some $\eps < \mu$, we have 
    $$
    \{f_\eps(\eta) : \eta \in \Lambda\} = \theta.
    $$
    (That is, $f_\eps \rest \Lambda$ is a surjection onto $\theta$.)
\end{enumerate}

So for some $u \in [2^\mu]^\Upsilon$ 
(with $\otp(u) = \Upsilon$, for simplicity), {let} $\{\eta_\alpha : \alpha \in u\}$ list the elements of $\Lambda$ without repetition.  

\mn
\textbf{Case 1:} $\mu = \mu^\theta$ and $\Upsilon \defeq \theta$.

For $\alpha \neq \beta \in u$, choose $\zeta_{\alpha,\beta} \in \mu$ such that $\nu_\alpha(\zeta_{\alpha,\beta}) \neq \nu_\beta(\zeta_{\alpha,\beta})$. Let 
$$
v \defeq \big\{ \zeta_{\alpha,\beta} : \alpha \neq \beta \in u \big\}.
$$

So clearly $v \in [\mu]^{\leq\theta}$ and $g : {}^v2 \to \theta$ both belong to $N$, where $g$ is the function which maps $\eta_\alpha \rest v \mapsto \otp(u \cap \alpha)$ for each $\alpha \in u$ and sends all other elements of ${}^v2$ to zero. Let $f : {}^\mu(2^\mu) \to \theta$ be defined by $\eta \mapsto g(\eta \rest v)$; clearly this is a member of $N$ as well. Hence $f = f_\eps$ for some $\eps < \mu$.

Now check.

\bn
\textbf{Case 2:} $\bfU_\theta(\mu) = \mu$ and $\Upsilon \defeq (2^\theta)^+ < 2^\mu$.

This is simply a special case of Case 3.

\bn
\textbf{Case 3:} $2^\theta \leq \mu$, and $\bfU_{[\sigma]^{<\theta}}(\mu) = \mu$ and $\Upsilon \defeq (2^{<\sigma})^+$ for some $\sigma \in [\theta,\mu]$.

We will try to choose $(\alpha_i,\beta_i,\zeta_i)$ by induction on $i < \sigma$ such that:
\begin{itemize}
    \item $\alpha_i \neq \beta_i$ are members of $u$.
\sn
    \item $\alpha_i,\beta_i \in u \setminus \{\alpha_j,\beta_j : j < i\}$
\sn
    \item $\zeta_i < \mu$
\sn
    \item $\eta_{\alpha_i}(\zeta_i) = 0$ and $\eta_{\beta_i}(\zeta_i) = 1$.
\sn
    \item If $j < i$ then $\eta_{\alpha_i}(\zeta_j) = \eta_{\beta_i}(\zeta_j)$. 
\end{itemize}

\mn
\textbf{Subcase A:} We succeed.

Let 
$W_0 \defeq \{\zeta_i : i < \sigma\}$ 
(so $W_0 \in [\mu]^\sigma$).
Now, using `$\bfU_{[\sigma]^{<\theta}}(\mu) = \mu$,' there exists 
$W \in [W_0]^\theta$ which
belongs to $N$, and we continue as in Case 1.

\mn
\textbf{Subcase B:} We get stuck at stage $i_*$, for some $i_* < \sigma$.

As $|u| = \Upsilon > 2^{|i_*|}$, there are $\alpha \neq \beta$ from $u$ such that 
$$
\eta_\alpha \rest \{\zeta_j : j < {i_*}\} = \eta_\beta \rest \{\zeta_j : j < {i_*}\}.
$$ 
This is an easy contradiction.

(Note that we can actually use $\Upsilon \defeq \sum\limits_{i < \sigma} (2^{|i|})^+$ instead of $(2^{<\sigma})^+$.)



\bn
\textbf{Case 4:} $\theta = \cf(\theta) < \mu$, $\mu$ is strong limit singular of cofinality 
$\neq \theta$, and $\Upsilon \defeq (2^{\theta+\cf(\mu)})^+$.

This is also a special case of Case 3.

\bn
\textbf{Case 5:} $\Upsilon \defeq \beth_\omega(\theta) \leq \mu$.

By \cite{Sh:460}, this is also covered by Case 3.
\end{PROOF}

\mn
We introduce the following relative of $\Sep = \Sep_1$ which will be used in this work.
\begin{definition}\label{x17}
1) Let $\Sep_2(\chi,\mu,\partial,\theta,\kappa,\clD)$ mean that there exists a sequence $\bar f = \LL f_{\eps,i} : \eps < \chi,\ i < \kappa\RR$ \emph{witnessing it}. By this, we mean that the following clauses hold.
\begin{enumerate}
    \item $\clD \in \up(\kappa)$
\sn
   \item $f_{\eps,i} : {}^\mu\partial \to \theta$
\sn
    \item If $\cP_i \subseteq {}^\mu\partial$ has\footnote{
        We may let this sequence be constant in $i$ --- if so we may write $\cP$ instead of $\olsi\cP = \LL \cP_i : i < \kappa\RR$.
    } 
    cardinality $<\partial^\mu$ (for $i < \kappa$), \underline{then} we can find a sequence
    $\bar\varrho = \LL \varrho_i : i < \kappa\RR$ such that:
    \begin{enumerate}
        \item $\varrho_i \in {}^\mu\theta$
\sn
        \item If $\bar\nu = \LL \nu_i : i < \kappa\RR \in \prod\limits_{i < \kappa} \cP_i$ \underline{then} there exist $\eps < \mu$ and $u \in \clD$ such that
        $$
        i \in u \Rightarrow f_{\eps,i}(\nu_i) = \varrho_i(\eps).
        $$
    \end{enumerate}
\end{enumerate}

\sn
2) If $\clD \defeq [\kappa]^\kappa$ we may omit it.
$\Sep_2(\mu,\theta,\kappa)$ will mean $\Sep_2(\mu,\mu,\theta,\theta,\kappa)$.
\end{definition}

\mn
Recalling Definition \ref{x5},
\begin{definition}\label{x23}
1)
We say that $\bfp$ has the $(D,\partial,\theta)$-$\olsi F$-\emph{BB-property}\footnote{
    Later, we will write `$\BB^0$-property.'
}
\underline{when} there exists a $\bfp$-$D$-$\olsi F$-BB-sequence, where:
\begin{enumerate}
    \item $\bfp$ is a $(\lambda,\kappa,\chi)$-BB-parameter.
\sn
    \item $D$ is a filter on $\lambda$.
\sn
    \item $\olsi F$ is a $(\bfp,\partial,\theta)$-coloring.
\end{enumerate}

\mn
2) We say that $\bfp$ has the $(D,\partial,\theta)$-{BB-property} \underline{when} it has $(D,\partial,\theta)$-$\olsi F$-{BB-property} for every $(\bfp,\partial,\theta)$-coloring $\olsi F$.

\mn
3) If $D$ is the club filter on $\lambda$, we may omit it.
\end{definition}

\mn
We now quote the main claim of the previous paper -- \cite[1.10\subref{d.6}]{Sh:775} -- but we will not use it here.
\begin{claim}\label{x26}
Assume
\begin{enumerate}[$(a)$]
    \item $\lambda \defeq \cf(2^\mu)$
\sn
    \item $D$ is a $\mu^+$-complete filter on $\lambda$ extending the club filter.
\sn
    \item $\kappa = \cf(\kappa) < \chi \leq \lambda$
\sn
    \item $\bfp = \big\LL (C_\delta,C_\delta') : \delta \in S \big\RR$ is a good $(\lambda,\kappa,\chi)$-$\BB$-parameter, where $S \in D$.
\sn
    \item $2^{<\chi} \leq 2^\mu$ and $\theta \leq \mu$.
\sn
    \item $\alpha < 2^\mu \Rightarrow \trp_\kappa^+(|\alpha|) \leq 2^\mu$ (By this we mean that every tree with $|\alpha|$-many nodes and $\kappa$ levels has $<2^\mu$-many $\kappa$-branches.)
\sn
    \item $\Sep_1(\mu,\theta)$.
\end{enumerate}

\sn
\underline{Then} $\bfp$ has the $(D,2^\mu,\theta)$-$\BB$-property. (Note that this means that possibly $\theta>2$; i.e.\ we have more than two colors.)
\end{claim}

\sn
\begin{remark}\label{x29}

1) If (e.g.) $\mu$ is strong limit singular, $\kappa \defeq \cf(\mu)$, $\lambda \defeq \cf(2^\mu)$, and $\kappa + \theta < \mu$, \underline{then} the only assumption in \ref{x26} which does not follow is clause $(f)$ (which does hold for many regular $\kappa < \mu$ by \cite{Sh:460}). For more, see \cite{Sh:829}.

Our aim here is to cover more cases of $\kappa$, and construct relatives of this property which are easier to use and have more applications.

\sn
2) By \cite[\S1]{Sh:420}, there are are many $S$ as required 
(usually from $\check I_\kappa[\lambda]$). Still, this restricts our choices.

\sn
3) `Good $\bfp$' is also a restriction, as the result covers fewer $S$-s. 
In fact, if $S \notin \check I_\kappa[\lambda]$ then there is no good $\bfp$ with $S_\bfp \defeq S$ {(as we cannot find $\LL C_\delta' : \delta \in S\RR$)}.
Another one of our goals is to eliminate this assumption.

\sn
4) But we would like to have parallel results using $\Sep_2$ or $\Sep_3$. (This will be done in \S2.)

\sn
5) An earlier definition of $\Sep_1(\mu,\theta)$ used `$\Upsilon < 2^\mu \vee \Upsilon \defeq 2^\mu \in \Reg$' instead of `$\Upsilon \leq 2^\mu \wedge \Upsilon \in \Reg$.' 

This was a natural generalization, because the notation is tailor-made for proofs which rely on induction on $\Upsilon \leq 2^\mu$.  As the $\Upsilon$ argument is an upper bound for the cardinality of some specific set, clearly $\Sep_1(\ldots,\Upsilon)$ implies $\Sep_1(\ldots,\Upsilon^+)$.  Moreover, as every successor ordinal is regular, if $\Upsilon < 2^\mu$ \underline{then} $\Upsilon^+ < 2^\mu \vee \Upsilon^+ \in \Reg$. However, we would have to rewrite existing proofs to match this new definition, and that would be more trouble than it's worth.
\end{remark}

\mn
\begin{claim}\label{x32}
Assume $\kappa$ is regular, $\mu > \theta = \theta^{<\kappa}$, and 
$\partial \in [2,2^\mu]$. If at least one of the following holds \underline{then} we have $\Sep_2(\mu,\mu,\partial,\theta,\kappa)$.
\begin{enumerate}[$(a)$]
    \item $\kappa \neq \cf(\mu)$, 
    $\alpha < \mu \Rightarrow |\alpha|^\kappa \le \mu$,
    and $\Sep_1(\mu,\partial,\theta)$.
\sn
    \item $\bfU_{\!\kappa}'(\mu) = \mu \geq \beth_\omega(\theta+\kappa)$
\sn
    \item $\bfU_{\!\kappa}'(\mu) = \mu$ and $\Sep_1(\mu,\partial,\theta)$.
\sn
    \item We have $\bfU_{[\sigma]^{<\theta}}(\mu) = \mu$ for some $\sigma \geq \theta$ with $\sigma^\theta \leq \mu$ and $(2^\sigma)^+ < 2^\mu$.
\end{enumerate}
\end{claim}

\begin{PROOF}{\ref{x32}}
\textbf{Case (a):}

Let $\bar f^\circ = \LL f_\eps^\circ : \eps < \mu\RR$ witness 
$\Sep_1(\mu,\partial,\theta)$ (hence $f_\eps^\circ$ is a function from ${}^\mu\partial$ to $\theta$). Let 
$$
\clF \defeq \{\nu \in {}^\kappa\!\mu : \rang(\nu) \text{ is a bounded subset of } \mu\}.
$$
Recalling Definition \ref{x8}(2), {let $\Upsilon$ be a regular cardinal 
$\leq 2^\mu$} such that\\ $\Sep_1(\mu,\mu,\partial,\theta,\Upsilon)$ holds.

By the assumption `$\alpha < \mu \Rightarrow |\alpha|^\kappa \leq \mu$,'
clearly $|\clF| = \mu$. Let $\LL \nu_\eps : \eps < \mu\RR$ list the members of $\clF$, and we shall define
\begin{enumerate}
    \item [$\circledast_1$] $\bar f = \LL f_{\eps,i} : \eps < \mu,\ i < \kappa\RR$, where $f_{\eps,i} \defeq f_{\nu_\eps(i)}^\circ$.
\end{enumerate}

It will suffice to prove that $\bar f$ witnesses $\Sep_2(\mu,\mu,\partial,\theta,\kappa)$. So let $\cP_i \subseteq {}^\mu\partial$ be of cardinality 
$< \partial^\mu = 2^\mu$ for $i < \kappa$, and we need to construct 
$\bar\varrho$ as in \ref{x17}(1)(C).

Fix $i < \kappa$, so by \ref{x8}(1)(B), for every $\rho \in {}^\mu\theta$ the set 
$$
\Sol_\rho \defeq \big\{ \nu \in {}^\mu\theta : (\forall \eps < \chi) \big[ f_\eps^\circ(\nu) \neq \rho(\eps) \big] \big\}
$$ 
has cardinality $< \Upsilon$.
As $|\cP_i| < 2^\mu$ and\footnote{
    Recall that we permit $\Upsilon \defeq 2^\mu$ if it is regular.
} $\Upsilon = \cf(\Upsilon) \le 2^\mu$, 
the set 
$\Lambda_i \defeq \bigcup\limits_{\nu \in \cP_i} \Sol_\nu$ has cardinality $< 2^\mu$, so we can choose $\varrho_i \in  {}^\mu\partial \setminus \Lambda_i$.

It will suffice to prove that $\LL\varrho_i : i < \kappa\RR$ is as promised. So let $\bar \nu = \LL \nu_i : i < \kappa\RR \in \prod\limits_{i < \kappa} \cP_i$, and we have to find $\eps < \mu$ {and $u \in [\kappa]^\kappa$} as promised in \ref{x17}(1)(C)(b).

For each $i < \kappa$, by our choice of $\varrho_i$ we know $\varrho_i \notin \Sol_{\nu_i}$. This means 
\begin{enumerate}
    \item [$\circledast_2$] There is $\eps_i < \mu$ such that 
    $f_{\eps_i}^\circ(\nu_i) = \varrho_i(\eps_i).$
\end{enumerate}

As $\cf(\mu) \neq \cf(\kappa)$, there exists $\zeta < \mu$ such that the set $u \defeq \{i < \kappa : \eps_i < \zeta\}$ has cardinality $\kappa$ (and even order type $\kappa$). 
\begin{enumerate}
    \item [$\circledast_3$] Let $\nu \in {}^\kappa\zeta$ be the sequence 
$$
\nu(i) \defeq 
\begin{cases}
    \eps_i &\text{if } i \in u\\
    0    &\text{otherwise,} 
\end{cases}
$$
and let $\eps < \mu$ be such that $\nu = \nu_\eps$.
\end{enumerate}
Now $\eps$ is as required.

Why? For every $i \in u$, we have
\begin{enumerate}
    \item [$\circledast_4$] $f_{\eps,i}(\nu_i) = f_{\nu_\eps(i)}^\circ(\nu_*) = f_{\eps_i}^\circ(\nu_i) = \varrho_i(\eps_i)$.
\end{enumerate}
[The first equality is the definition from $\circledast_1$, the second holds by the choice of $\nu$ in $\circledast_3$, and the third by the choice of $\eps_i$ in $\circledast_2$.]

\mn
\textbf{Case (b):} Assume $\bfU_{\!\kappa}(\mu) = \mu \geq \chi \defeq \beth_\omega(\theta+\kappa)$.

By Case (e) of \ref{x20} this implies $\Sep_1(\mu,\theta)$, so 
let $\bar f^\circ = \LL f_\eps^\circ : \eps < \mu\RR$ be a witness.
Let $\clF \subseteq {}^\kappa\mu$
be of cardinality $\mu$ witnessing $\bfU_{\!\kappa}'(\mu) = \mu$.

The rest is as in the proof of Case (a), except that in the end we choose $u \in [\kappa]^\kappa$ 
and $\eps < \mu$ {together} such that
$$
(\forall i \in u)\big[\nu_\eps(i) = \eps_i \big]
$$
(which is possible by our choice of $\clF$).

\mn
\textbf{Case (c):} Like Case (b).

\mn
\textbf{Case (d):} Similarly, using \ref{x20}(c).
\end{PROOF}

\mn
\begin{claim}\label{x35}
$1)$ $\mu = \mu^{\kappa+\theta}$ implies $\Sep_2(\mu,\mu,2^\mu,\theta,\kappa,\{\kappa\})$.

\sn
$2)$ Suppose $\bfU_\kappa(\mu) = \kappa$ and
$\bfU_{[\sigma]^{<\theta}}(\mu) = \mu$ for some $\sigma \geq \theta$ with $\sigma^\theta \leq \mu$ and $(2^\theta)^+ < 2^\mu$.

\underline{Then} we have $\Sep_2(\mu,\mu,2^\mu,\theta,\kappa,[\kappa]^\kappa)$.
\end{claim}

\begin{PROOF}{\ref{x35}}
1) Like the proof of Case (a) of \ref{x32}, using $\clF \defeq {}^\kappa\mu$.

\sn
2) Use \ref{x32} Case (d) and the proof of Case (a).
\end{PROOF}

\mn
\begin{definition}\label{x38}
Assume $\lambda,\chi,\mu,\partial,\theta,\kappa$ are cardinals (with $\partial$ and $\theta$ possibly finite), and $\clD \in \up(\kappa)$.

\mn
1) Let $\Sep_3(\lambda;\chi,\mu,\partial,\theta,\kappa,\clD)$ mean that there exist 
$$
\bar f = \LL f_{\eps,i} : \eps < \chi,\ i < \kappa\RR
$$ 
and $\olsi\cP = \LL \cP_\xi : \xi < \lambda\RR$ such that the following all hold.
\begin{enumerate}
   \item $f_{\eps,i} : {}^\mu\partial \to \theta$
\sn
    \item $\cP_\xi \subseteq {}^\mu\partial$, and $\olsi \cP$ is strictly $\subset$-increasing in $\xi$ with union ${}^\mu\partial$.
\sn
    \item If $\xi < \lambda$ then we can
    we can find a sequence
    $\bar\varrho = \LL \varrho_i : i < \kappa\RR$ such that:
    \begin{enumerate}
        \item $\varrho_i \in {}^\mu\theta$
\sn
        \item If $\bar\nu = \LL \nu_i : i < \kappa\RR \in {}^\kappa(\cP_\xi)$ 
        \underline{then} there exist $\eps < \chi$ 
        and $u \in \clD$ such that
$$
        i \in u \Rightarrow f_{\eps,i}(\nu_i) = \varrho_i(\eps).
$$
    \end{enumerate}
\end{enumerate}

\sn
2) $\Sep_4(\lambda;\chi,\mu,\partial,\theta,\kappa,\clD)$ is defined similarly, except that $\olsi \cP$ is redefined as a covering of ${}^\kappa({}^\mu\partial)$ 
(so each $\cP_\xi \subset {}^\kappa({}^\mu\partial)$). Clause (C) remains the same, except that the antecedent to (C)(b) becomes ``If $\bar\nu = \LL \nu_i : i < \kappa\RR \in \cP_\xi$ . . .' for consistency.

\sn
3) {Again, the default value of $\clD$ is $[\kappa]^\kappa$.} 
For $\iota = 3,4$,
$\Sep_\iota(\lambda;\mu,\theta,\kappa)$ will mean $\Sep_\iota(\lambda;\mu,\mu,\theta,\theta,\kappa)$. (That is, $\chi \defeq \mu$ and $\partial \defeq \theta$.)
\end{definition}

\mn
\begin{observation}\label{x41}
$1)$ \emph{\textbf{[Monotonicity:]}} If $\chi_1 \leq \chi_2$, $\mu_1 \leq \mu_2$, 
$\partial_1 \geq \partial_2$, $\theta_1 \geq \theta_2$, and $\clD_1 \supseteq \clD_2$, \underline{then}
$$
\Sep_3(\lambda;\chi_1,\mu_1,\partial_1,\theta_1,\kappa,\clD_1) \Rightarrow \Sep_3(\lambda;\chi_2,\mu_2,\partial_2,\theta_2,\kappa,\clD_2).
$$

\sn
\emph{1A)} Similarly for $\Sep_4$.

\sn
\emph{1B)} Similarly for $\Sep_1$; i.e.\ if $\chi_1 \leq \chi_2$, $\mu_1 \leq \mu_2$, 
$\partial_1 \geq \partial_2$, $\theta_1 \geq \theta_2$, and $\Upsilon_1 \leq \Upsilon_2 \leq 2^{\mu_1}$,\footnote{
    Again, with equality only if $2^{\mu_1}$ is regular.
} \underline{then}
$$
\Sep_1(\chi_1,\mu_1,\partial_1,\theta_1,\Upsilon_1) \Rightarrow \Sep_1(\chi_2,\mu_2,\partial_2,\theta_2,\Upsilon_2).
$$

\sn
$2)$ \emph{\textbf{[Connection to $\Sep_2$:]}} Assume $\chi,\mu,\partial,\theta,\clD$ are as in \emph{\ref{x38}}, and $\cf(\partial^\mu) > \kappa$.\footnote{
    This is needed in the second $\Rightarrow$.
}
\underline{Then}
$$
\Sep_2(\chi,\mu,\partial,\theta,\kappa,\clD) \Rightarrow \Sep_3(\cf(\partial^\mu);\chi,\mu,\partial,\theta,\kappa,\clD) \Rightarrow \Sep_4(\cf(\partial^\mu);\chi,\mu,\partial,\theta,\kappa,\clD)
$$
\end{observation}

\begin{PROOF}{\ref{x41}}
1) Read the definition. 
(1A) and (1B) are similar.

\sn
2) \textbf{The first implication:}

Let $\LL \eta_\alpha : \alpha < \partial^\mu\RR$ list ${}^\mu\partial$, let $\LL \Upsilon_\zeta : \zeta < \lambda\RR$ be an increasing sequence of ordinals with limit $\partial^\mu$, and let $\cP_\zeta \defeq \{\eta_\alpha : \alpha < \Upsilon_\zeta\}$.

Now check.

\mn
\textbf{The second implication:}

Let $\big\LL \cP_\xi : \xi < \cf(\partial^\mu) \big\RR$ exemplify 
$$
\Sep_3(\cf(\partial^\mu);\chi,\mu,\partial,\theta,\kappa,\clD)
$$
(that is, it is as in Definition \ref{x38}(1).) Now $\big\LL {}^\kappa (\cP_\xi) : \xi < \cf(\partial^\mu) \big\RR$ is as required in \ref{x38}(2). (Note that here we need the assumption `$\cf(\partial^\mu) > \kappa$.')
\end{PROOF}

\mn
\begin{conclusion}\label{x47}
We have $\Sep_3(\lambda;\mu,\mu,2^\mu,\theta,\kappa,\clD)$ \underline{when}
\begin{enumerate}[$(A)$]
    \item $\kappa = \cf(\kappa)$
\sn
    \item The triple $(\mu,2^\mu,\theta)$ satisfies at least one of the conditions in \emph{\ref{x20}}.
\sn
    \item $\bfU_\clD(\mu) = \mu$ 
\sn
    \item $\lambda \defeq \cf(2^\mu)$. 
\end{enumerate}
\end{conclusion}

\begin{PROOF}{\ref{x47}}
By the conclusion of \ref{x20}, we have $\Sep_1(\mu,\mu,2^\mu,\theta,\Upsilon)$ for some $\Upsilon$. 
Using clause $(C)$ and repeating the proof of \ref{x32}(a), we have $\Sep_2(\mu,\mu,2^\mu,\theta,\kappa,\clD)$.
With that and \ref{x41}(2), we get our conclusion.
\end{PROOF}

-

\mn
\begin{remark}\label{x50}
1) We may need to sort out when `$\chi \ne \mu$' is actually needed in $\Sep_\iota$.

\sn
2) We may also consider the definition below.
\end{remark}

\mn\begin{definition}\label{x53}
Let $\kappa \leq \mu$ and $\bbS \subseteq \clP(\clP(\kappa))$.

\sn
1) We define $\bfU_\bbS(\mu)$ to be 
$$
\min\!\big\{|\clU| : \clU \subseteq [\mu]^\kappa,\ (\forall g \in {}^\kappa\!\mu)(\exists \clA \in \bbS)(\forall u \in \clA)\big[ \rang(g \rest u) \in \clU\big]\big\}.
$$

\sn
2) $\bfU_\bbS'(\mu)$ means
$$
\min\!\big\{ |\clF| : \clF \subseteq {}^\kappa\!\mu \text{ and } (\forall g \in {}^\kappa\!\mu)(\exists \clA \in \bbS)(\forall u \in \clA)(\exists f \in \clF)[f \rest u = g\rest u]  \big\}.
$$
\end{definition}

\mn
\begin{question}\label{x60}
Can we prove existence results for $\Sep_\iota(\lambda;\mu,\theta,\kappa)$ for ($\iota = 3,4$ and) $\lambda \in (\mu,2^\mu] \setminus \{\cf(2^\mu)\}$ regular? Can we disprove them?
\end{question}

\mn
Well, as in many cases, we have independence results. E.g.
\begin{claim}\label{x63}
$1)$ Assume $\mu = \mu^{<\mu} < \chi = \chi^\mu$. Let $\bbP \defeq \Cohen_{\mu,\chi}$ (i.e.\ the forcing adding $\chi$-many $\mu$-Cohens). 

\underline{Then} in $\bfV^\bbP$, for ever regular $\lambda \in [\mu^+,\chi] \setminus \{\cf(\chi)\}$, we have $\Sep_3(\lambda; \mu,\theta,\kappa)$.

\sn
$2)$ If $\theta \leq \mu < \lambda = \cf(\lambda) \leq 2^\mu$ and $\bar\eta = \LL\eta_\xi : \xi < \lambda\RR \subseteq {}^\mu\theta$ is a $\mu$-Luzin sequence, \underline{then} $\Sep_3(\lambda;\mu,\mu,2^\mu,\theta,\kappa,\{\kappa\})$.

Recall that `$\bar\eta$ is a $\mu$-Luzin sequence' means that for every meagre $\mu$-Borel set $\bfB \subseteq {}^\mu\theta$ there exists $\alpha_* < \lh(\bar\eta)$ such that
$$
\alpha_* < \alpha < \lh(\bar\eta) \Rightarrow \eta_\alpha \notin \bfB.
$$
\end{claim}

\begin{PROOF}{\ref{x63}}
Straightforward.
\end{PROOF}

\mn
\begin{claim}\label{x66}
Let $\mu = \mu^{<\mu} < \lambda = \cf(\lambda) < \chi \leq 2^\mu$ and assume (e.g.) the forcing axiom $\mathbf{Ax}_{\mu,\chi}((1)_c^+,(2)_c^\eps)$ from \cite{Sh:1036}. Let $S \defeq \{\delta \in S_\kappa^\lambda : \mu^\kappa \divides \delta\}$.\footnote{
    Here we mean $\mu^\kappa$ as exponentiation of ordinals.
}

\underline{Then}
\begin{enumerate}[$(A)$]
    \item For $\kappa = \cf(\kappa) < \mu$, there exists $\olsi C = \LL C_\delta : \delta \in S\RR$ such that:
    \begin{enumerate}[$(a)$]
        \item $C_\delta$ is an unbounded subset of $\delta$ of order type $\kappa$.
\sn
        \item $\olsi C$ is $\mu^+$-free (see \emph{\ref{yfree}}).  
    \end{enumerate}
    (I.e.\ if $u \in \big[{S_\kappa^\lambda} \big]^{<\mu}$ then there exists $f \in \prod\limits_{\delta \in u} C_\delta$ such that $\LL C_\delta \setminus f(\delta) : \delta \in u\RR$ is a sequence of pairwise disjoint sets.)
\sn
    \item Any $\olsi C$ as in part $(A)$ has $\mu$-\emph{uniformization}.

    That is, if $\bar f = \LL f_\delta : \delta \in {S_\kappa^\lambda}\RR$ with 
    $f_\delta \in {}^{C_\delta}\mu$
    \underline{then} there exists 
    $f \in {}^\lambda\mu$ such that 
    $$
    (\forall \delta \in {S_\kappa^\lambda}) \big[ f_\delta \subseteq^* f \big].
    $$
    (By \emph{this} we mean $\big|\{ \alpha \in C_\delta : f_\delta(\alpha) \neq f(\alpha)\} \big| < \kappa$.)
\end{enumerate}
\end{claim}

\begin{PROOF}{\ref{x66}}
\textbf{Clause (A):} 

Choose $\LL C_\delta^\circ : \delta \in S\RR$ such that:
\begin{enumerate}
    \item [$(*)_\delta$] If $\delta \in S$ (hence $\mu^\kappa \divides \delta$) then $C_\delta^\circ$ is an unbounded subset of $\delta$ of order type $\mu^\kappa$. 
\end{enumerate}

Now we define a forcing notion $\bbQ$.
\begin{enumerate}
    \item[$(*)_\bbQ$]
    \begin{enumerate}
        \item  $p \in \bbQ$ \underline{iff} $p$ is of the form 
        $\LL C_\delta : \delta \in u\RR$, where:
        \begin{enumerate}
            \item $u \in [S]^{<\mu}$
\sn
            \item $C_\delta \subseteq C_\delta^\circ$ is also unbounded in $\delta$.
\sn
            \item $\otp(C_\delta) = \kappa$
        \end{enumerate}
\sn
        \item $\bbQ \models `p \leq q$' \underline{iff} $u_p \subseteq u_q$ and 
        $C_{p,\delta} = C_{q,\delta}$ for all $\delta \in u_p$.
    \end{enumerate}
\end{enumerate}

Clearly $\bbQ$ is strategically $\alpha$-complete for all $\alpha < \mu$ (i.e.\ \cite[0.3(A)(1)$_c$\subref{x2}]{Sh:1036}). Moreover, this still holds for $\alpha = \mu$, and we have version (2)$_b$ of the $\mu^+$-cc (from \cite[0.3(B)\subref{x2}]{Sh:1036}).

\mn
\textbf{Clause (B):} Similarly.
\end{PROOF}

\bn
The following observation will give us some bounds.
\begin{definition}\label{x70}
For $\theta \in \kappa \cap \Reg$, let 
$$
\inv_\sigma(\kappa) \defeq \min\!\big\{ |\clA| : \clA \subseteq [\kappa]^\kappa \text{ and } (\forall h \in {}^\kappa\sigma)(\exists i < \sigma)(\exists A \in \clA)[ \rang(h \rest A) \subseteq i] \big\}.
$$
\end{definition}

\bn
\begin{observation}\label{x73}
$1)$ If $\sigma = \cf(\sigma) < \kappa = \kappa^{<\kappa} < \lambda = \lambda^\kappa$ and $\bbQ$ is the forcing which adds $\kappa$-many $\sigma$-Cohens, then $\Vdash_\bbQ ``\inv_\sigma(\kappa) = \lambda = 2^\kappa$".

\mn
$2) $If $\sigma = \cf(\sigma) \leq \partial = \partial^\kappa$ and $\bbQ$ is a $\lepref{\partial}$-complete forcing notion with {calibre} $\mu$ which does not collapse cardinals, then in $\bfV^\bbQ$ we have $\inv_\sigma(\kappa)^{\bfV^\bbQ} \leq (2^\kappa)^\bfV$. 
\end{observation}

\newpage
\section{The Black Box property}\label{S2}

\begin{definition}\label{a2}
1) We say $\bfp = (\big\LL (C_\delta,C_\delta',\olsi C_\delta) : \delta \in S\big\RR,\clD)$ is a \emph{$(\lambda,\kappa,\chi)$-$\BB^1$-parameter} \underline{when} 
\begin{enumerate}
    \item 
    \begin{enumerate}
        \item $\lambda > \kappa$ are regular.
\sn
        \item $S$ is a stationary subset of $S_\kappa^\lambda$.
    \end{enumerate}
\sn
    \item $C_\delta \subseteq \delta$ with $\sup(C_\delta) = \delta$.
\sn
    \item 
    \begin{enumerate}
        \item $C_\delta' \subseteq \delta$, $\sup(C_\delta') = \delta$, and $\otp(C_\delta') = \kappa$. (We do not require that $C_\delta$ \emph{or} $C_\delta'$ be closed in $\delta$.)
\sn
        \item Let $\bar \beta_\delta = \LL \beta_{\delta,i} : i < \kappa\RR$ list $C_\delta'$ in increasing order.
    \end{enumerate}
    
\sn
    \item 
    \begin{enumerate}
        \item $\olsi C_\delta = \LL C_{\delta,i} : i < \kappa\RR$
\sn
        \item $C_{\delta,i} \subseteq C_\delta \cap \beta_{\delta,i}$
\sn
        \item $|C_{\delta,i}| < \chi$
    \end{enumerate}
\sn
    \item For all $\alpha < \lambda$, the set 
    $$
    \big\{ C_{\delta,i} : \delta \in S,\ i < \kappa,\ \beta_{\delta,i} = \alpha \big\}
    $$
    has cardinality $<\lambda$.
\sn
    \item $\clD \in \up(\kappa)$.
\end{enumerate}

\sn
1A) Above, replacing $\BB^1$ by $\BB^2$ means strengthening clause (1)(D)(b) to 
\begin{enumerate}
    \item [(b$'$)] $C_{\delta,i} \defeq C_\delta \cap \beta_{\delta,i} \setminus \beta_{\delta,<i}$.
\end{enumerate}

\sn
1B) Replacing $\BB^1$ by $\BB^3$ means strengthening clause (1)(D)(b) to 
\begin{enumerate}
    \item [(b$''$)] $C_{\delta,i} \defeq C_\delta \cap \beta_{\delta,i}$.
\end{enumerate}

\sn
2) We say $\olsi F = \LL F_{\delta,i} : \delta \in S,\ i < \kappa\RR$ is a \emph{continuous $(\bfp,\partial,\theta)$-$\BB^1$-coloring} \underline{when} ($S = S_\bfp$ and) 
\begin{enumerate}
    \item $F_{\delta,i} : {}^{C_{\delta,i}}\partial \to \theta$ for $\delta \in S$ and $i < \kappa$ (recalling 
    $C_{\delta,i} \subseteq C_\delta \cap \beta_{\delta,i}$). 
\sn
    \item For $\beta < \lambda$, the set $\clF_\beta \defeq \{F_{\delta,i} : \delta \in S,\ i < \kappa,\ \beta_{\delta,i} = \beta\}$ has cardinality $<\lambda$. 
\end{enumerate}

\sn 
2A) For $\olsi F$, $\partial$, and $\theta$ as above and $D$ a filter on $\lambda$, we say 
$\bar c = \LL c_{\delta,i} : \delta \in S,\ i < \kappa\RR \in {}^{S\times \kappa} \theta$ is a 
$\bfp$-$D$-$\olsi F$-\emph{$\BB^1$-sequence} \underline{when} for every 
$\eta \in {}^\lambda\partial$ the set 
$$
\big\{\delta \in S : (\exists^\clD i < \kappa)[F_{\delta,i}(\eta \rest C_{\delta,i})  = c_{\delta,i}] \big\}
$$ 
is a member of $D^+$ (and in the default case, a stationary subset of $\lambda$).\footnote{
    Naturally, $(\exists^\clD i < \kappa)\varphi_i$ is shorthand for $(\exists u \in \clD)(\forall i \in u)\varphi_i$.
}


\sn 
3) We say that $\gamma,\delta \in S$ are $\bfp$-\emph{similar} \underline{when}:
\begin{enumerate}[$\bullet_1$]
    \item $\otp(C_\gamma) = \otp(C_\delta)$ (Recall that we only demanded $\otp(C_\gamma') = \otp(C_\delta') = \kappa$.)
\sn
    \item $i < \kappa \Rightarrow \otp(C_{\gamma,i}) = \otp(C_{\delta,i})$.
\end{enumerate}

\sn 
4) We say $\olsi F = \LL F_\delta : \delta \in S\RR$ is a \emph{uniform $(\bfp,\partial,\theta)$-coloring} \underline{when} the implication `(A) $\Rightarrow$ (B)' holds, where:
\begin{enumerate}
    \item 
    \begin{enumerate}
        \item $\delta_1$ and $\delta_2$ are $\bfp$-similar.
\sn
        \item $f_\ell : C_{\delta_\ell} \to \theta$ for $\ell = 1,2$.
\sn
        \item If $\gamma_\ell \in  C_{\delta_\ell}$ for $\ell =1,2$, then 
        $$
        \otp(\gamma_1 \cap C_{\delta_1}) = \otp(\gamma_2 \cap C_{\delta_2}) \Rightarrow f_1(\gamma_1) = f_2(\gamma_2).
        $$
    \end{enumerate}
    
\sn
    \item $F_{\delta_1}(f_1) = F_{\delta_2}(f_2)$.
\end{enumerate}

\sn 
5) `$\LL F_{\delta,i} : \delta \in S,\ i < \kappa\RR$ is a \emph{uniformly continuous} $(\bfp,\partial,\theta)$-coloring' is defined similarly, but we replace (4)(B) with the demand 
\begin{enumerate}
    \item [(B)$'$] $i < \kappa \Rightarrow F_{\delta_1,i}(f_1 \rest C_{\delta_1,i}) = F_{\delta_2,i}(f_2 \rest C_{\delta_2,i}).$
\end{enumerate}

\sn 
6) When we write $\sigma$-uniform instead of uniform, this means that in clause (4)(A)(a) we replace `$\bfp$-similar' by `$\clE$-equivalent' for some equivalence relation $\clE$ on $S$ with $\leq\sigma$ equivalence classes satisfying 
$$
\gamma\ \clE\ \delta \Rightarrow \text{[$\gamma$ is $\bfp$-similar to $\delta$].}
$$
\end{definition}

\mn
\begin{remark}\label{a3}
1) Regarding $\BB^2$ --- the idea is that  $C_\delta$ and the $C_{\delta,i}$-s are defined in terms of $C_\delta'$ and 
$\beta_{\delta,i}$ (recalling $C_\delta' = \{\beta_{\delta,i} : i < \kappa\}$). 

On the one hand, we can choose $C_\delta \defeq \delta $, in which case 
$\chi = \lambda$ and we may choose $C_{\delta,i} \defeq [\beta_{\delta,<i},\beta_{\delta,i}]$. 

On the other hand, we may choose 
$C_\delta' \defeq C_\delta$ and 
$$
C_{\delta,i} \defeq 
\begin{cases}
    \{\beta_{\delta,i-1}\} &\text{if $i$ is successor}\\
    \varnothing &\text{otherwise.}
\end{cases}
$$

In both cases we get clause \ref{a2}(1)(E); that is,
$$
\beta < \lambda \Rightarrow \big| \{C_{\delta,i} : \delta \in S,\ i < \kappa,\ \beta_{\delta,i} = \beta\} \big| < \lambda.
$$
This will be used in clause $(*)_4$(d) in the proof of \ref{a8}.

\sn
2)  Uniformity (defined in \ref{a2}(4)-(6)) is used only in \ref{a37}.

\sn
3) Note that in Lemma \ref{a8} 
we may choose $\clD = \clD_\bfp \defeq \{\kappa\}$ --- {the best, the desired case.}

But then the assumption on $\Sep_3$ is stronger, so it is better to apply it to $\clD \defeq [\kappa]^\kappa$; then we shall be able to upgrade the conclusion
in \ref{a28}, \ref{a31}.

\sn
4) For the $\BB^1$ version, we may omit \ref{a2}(3)$\bullet_1$.
\end{remark}

\mn
\begin{definition}\label{a5}
1) We say that $\bfp$ has the \emph{continuous $(D,\partial,\theta)$-$\BB^1$-property} \underline{when} it has the $(D,\partial,\theta)$-$\olsi F$-$\BB^1$-property (see Definition \ref{a2}(2A)) for every continuous $(\bfp,\partial,\theta)$-coloring $\olsi F$.

\sn
By this, we mean that the implication `(A) $\Rightarrow$ (B)' holds, where:
\begin{enumerate}
    \item 
    \begin{enumerate}
        \item $\bfp$ is a $(\lambda,\kappa,\chi)$-BB-parameter.
\sn
        \item $D$ is a filter on $\lambda$.
\sn
        \item $\olsi F$ is a \emph{continuous} $(\bfp,\partial,\theta)$-coloring.
    \end{enumerate}
\sn
    \item There exists a $\bfp$-$D$-$\olsi F$-$\BB^1$-sequence.
\end{enumerate}

\sn
2) Again, if $D$ is the club filter on $\lambda$ plus $S_\bfp$, then we may omit it.
\end{definition}

\mn
\begin{discussion}\label{a6}
The next claim is related to \ref{x26}. We restrict ourselves to continuous colorings, but we gain by omitting demand \ref{x26}$(f)$, which restricted 
$\kappa$ (that is, the cofinality of members of $S$). 
Also, if $\lambda = \lambda^{<\lambda}$ (hence $\lambda = 2^\mu$) then we may choose $\chi \defeq \lambda$.
\end{discussion}

\mn
\begin{lemma}\label{a8}
We have $`(A) \Rightarrow (B)$\emph{'}, where
\begin{enumerate}[$(A)$]
    \item 
    \begin{enumerate}[$(a)$]
        \item $\mu <\lambda = \cf(\lambda) = \cf(2^\mu)$
\sn
        \item $D$ is a $\mu^+$-complete filter on $\lambda$ extending the club filter.
\sn
        \item $\kappa = \cf(\kappa) < \chi \leq \lambda$
\sn
        \item $\bfp = (\big\LL (C_\delta,C_\delta',\olsi C_\delta) : \delta \in S \big\RR,\clD)$ is a $(\lambda,\kappa,\chi)$-$\BB^1$-parameter, where $S \in D^+$.
\sn
        \item $\theta \leq \mu$ and $\theta^{<\chi} \leq 2^\mu$.
\sn
        \item $\Sep_3(\lambda;\mu,\mu,\partial,\theta,\kappa,\clD)$.
    \end{enumerate}
\sn
    \item $\bfp$ has the continuous $(D,2^\mu,\theta)$-$\BB^1$-property.
\end{enumerate}
\end{lemma}

\sn
\begin{remark}\label{a11}
1) If $C_{\delta_1,\eps_1} = C_{\delta_2,\eps_2} \Rightarrow F_{\delta_1,\eps_1} = F_{\delta_2,\eps_2}$, \underline{then} in the proof below, we can replace $\clF_\beta$ by $\clC_\beta \defeq \{C_{\delta,i} : \delta \in S,\ i < \kappa,\ \beta_{\delta,i} = \beta\}$.

\mn
2) The main case is $\partial \defeq 2^\mu$.
\end{remark}

\begin{PROOF}{\ref{a8}}
First,
\begin{enumerate}
    \item [$(*)_1$] Let $\olsi F = \LL F_{\delta,i} : \delta \in S,\ i < \kappa\RR$ be a continuous $(\bfp,2^\mu,\theta)$-coloring. It will suffice to show that there is a $\bfp$-$D$-$\olsi F$-sequence.
\end{enumerate}

Now by assumption $(A)(f)$,
\begin{enumerate}
    \item [$(*)_2$] Let $\olsi f = \LL f_{\eps,i} : \eps < \mu,\ i < \kappa\RR$ and $\olsi \cP = \LL \cP_\xi : \xi < \lambda\RR$ exemplify 
    $$
    \Sep_3(\lambda;\mu,\mu,\partial,\theta,\kappa,\clD).
    $$
    \item [$(*)_3$] For each $\delta \in S$, choose $\bar\varrho^\delta = \LL \varrho_i^\delta : i < \kappa\RR \in {}^\kappa({}^\mu\theta)$ as in Definition \ref{x38}(1)(C).
    
    (Recall that $\delta \in S \Rightarrow \delta < \lambda$, so the $\delta$ here corresponds to the subscript $\xi$ in the definition.)
\end{enumerate}

Now for every $\eps < \mu$ we suggest a possible $\bfp$-$D$-$\olsi F$-sequence $\bar c_\eps$:
\begin{enumerate}
    \item [$(*)_4$] $\bar c_\eps = \LL c_{\eps,\delta,i} : \delta \in S,\ i < \kappa\RR$, where $c_{\eps,\delta,i} \defeq f_{\eps,i}(\varrho_i^\delta) \in \theta$.
\sn
    \item [$(*)_5$] Let $\clF_\beta \defeq \{F_{\delta,i} : \delta \in S,\ i < \kappa,\ \beta_{\delta,i} = \beta\}$, so
    \begin{itemize}
        \item $\dom(F_{\delta,i}) = {}^{(C_{\delta,i})}({}^\mu\partial)$ (That is, the set of functions $f : C_{\delta,i} \to {}^\mu\partial$.)
\sn
        \item $\rang(F_{\delta,i}) \subseteq \theta$
\sn
        \item If $F = F_{\delta,i} \in \clF_\beta$, then we may write $C_F$ as a well-defined shorthand for $C_{\delta,i}$.
    \end{itemize}
\sn
    \item [$(*)_6$] Let $\clC' \defeq \bigcup\limits_{\delta \in S} C_\delta'$.
\end{enumerate}

If for some $\eps < \mu$ the sequence $\bar c_\eps$ is as required (i.e.\ it is a $\bfp$-$D$-$\olsi F$-sequence) then we are done.
So toward contradiction, assume
\begin{enumerate}
    \item [$\boxplus_{6.1}$] For all $\eps < \mu$ we can choose $\eta_\eps \in {}^\lambda({}^\mu\partial)$ and $E_\eps \in D$ such that
    $$
    \delta \in E_\eps \Rightarrow \big\{i < \kappa : c_{\eps,\delta,i} = F_{\delta,i}(\eta \rest C_{\delta,i}) \big\} \notin \clD.
    $$

    \item [$(*)_7$] For every {$\beta \in \clC'$} and $F \in \clF_\beta$, let
    \begin{itemize}
        \item  $\rho_{\beta,F} \defeq \LL F(\eta_\eps \rest C_F) : \eps < \mu\RR$ (so $\rho_{\beta,F} \in {}^\mu\theta$).
\sn
        \item ${j_{\beta,F}} \defeq \min\{\zeta < \lambda : \rho_{\beta,F} \in \cP_\zeta\}$
\sn
        \item $\zeta_\beta \defeq \sup\{j_{\beta,F} : F \in \clF_\beta\} < \lambda$
    \end{itemize}
\end{enumerate}
[Why can we do this? Because $|\clF_\beta| < \lambda = \cf(\lambda)$ by \ref{a2}(2)(B).]

\begin{enumerate}
    \item [$(*)_8$] Let $E_* \defeq \{\delta < \lambda : \delta \text{ limit, and } \beta < \delta \Rightarrow \zeta_\beta < \delta\}$.
    \begin{enumerate}
        \item $E_*$ is a club of $\lambda$.
\sn
        \item $E \defeq E_* \cap \bigcap\limits_{\eps < \mu} E_\eps \in D$
    \end{enumerate}
\end{enumerate}
[Why {is this in the filter}? We declared each $E_\eps \in D$ by $\boxplus_{6.1}$, and clause (a) holds by our construction. By assumption $(A)(b)$, $D$ contains all clubs and is $\mu^+$-complete.]
\begin{enumerate}    
    \item [$(*)_9$] For every $\delta \in E$ there exists $\eps = \eps_\delta < \mu$ such that 
    $$
    \{i < \kappa : F_{\delta,i}(\eta_\eps \rest C_{\delta,i}) = c_{\eps,\delta,i} \} \in \clD.
    $$ 
\end{enumerate}
[Why? By the choice of $\bar\varrho^\delta$ --- that is, by the definition of $\Sep_3$ --- because\\ $\eps < \mu \Rightarrow \eta_\eps \in \cP_\delta$.]
\begin{enumerate}    
    \item [$(*)_{10}$] For some $\eps < \mu$, 
    $$
    A_\eps \defeq \{ \delta \in E \cap S : \eps_\delta = \eps\} \in D^+.
    $$
\end{enumerate}
[Why? By $(*)_9$ and the fact that $D$ is $\mu^+$-complete.]

\sn
But this is a contradiction.
\end{PROOF}

\mn
\begin{remark}\label{a20}
It is nice to successfully predict {the values of} $\LL F_{\delta,i}(\eta) \rest C_{\delta,i} : i \in u\RR$ on some $u \in [\kappa]^\kappa$, but it would be better to succeed for $u = \kappa$. 

One possibility: what if we just assume $\theta = \theta^{<\kappa}$, and for each $u \subseteq \kappa$ we define $\bfp_{[u]}$ by $(S_{\bfp_{[u]}},\olsi C_{\bfp_{[u]}}) \defeq (S_\bfp,\olsi C_\bfp)$, but 
$$
C_{\bfp_{[u]},\delta}' \defeq \{\alpha \in C_{\bfp,\delta}' : \otp(C_{\bfp,\delta}' \cap \alpha) \in u\}?
$$
\underline{Or} use a regressive function $h : u \to \kappa$? Something close is done below.
\end{remark}

\mn
\begin{definition}\label{a22}
Let $\bfp$ be a $(\lambda,\kappa,\chi)$-{$\BB^1$}-parameter.

For $A \in [\kappa]^\kappa$, we define a $(\lambda,\kappa,\chi)$-$\BB^1$-parameter 
$$
\bfp_{\!A} = \bfp[A] = (\big\LL C_{\!A,\delta},C_{\!A,\delta}',\olsi C_{\!A,\delta} : \delta \in S\big\RR,\clD_A)
$$
by
\begin{itemize}
    \item $C_{\!A,\delta} \defeq C_\delta$
\sn
    \item $C_{\!A,\delta}' \defeq \{\beta \in C_\delta' : \otp(C_\delta' \cap \beta) \in A\}$.
\sn
    \item $\olsi C_{\!A,\delta} = \LL C_{\!A,\delta,i} : i < \kappa\RR$ is defined by $C_{\!A,\delta,i} \defeq C_{\delta,h_A^{\supminus1}(i)}$, where 
    $h_A : A \to \kappa$ is the function $i \mapsto \otp(A \cap i)$.\footnote{
        Note that this function is invertible as it is strictly increasing.
    }
\sn
    \item $\clD_A \defeq \{\kappa\}$ 
\end{itemize}
\end{definition}

\sn
\begin{observation}\label{a25}
$1)$ $\bfp_{\!A}$, as defined above, is indeed a $(\lambda,\kappa,\chi)$-$\BB^1$-parameter.

\sn
$2)$ If $\bfp$ is good then so is $\bfp_{\!A}$.
\end{observation}

\sn
\begin{claim}\label{a28}
Assume $\bfp$ is a $(\lambda,\kappa,\chi)$-$\BB^1$-parameter, 
$\clD_\bfp \defeq [\kappa]^\kappa$, $D$ is a $(2^\kappa)^+$-complete filter 
on $\lambda$, 
$\partial = \partial^{2^\kappa}$, and $\theta = \theta^{2^\kappa}$.

\sn
\underline{Then} $\bfp$ has the continuous $(D,\partial,\theta)$-$\BB^1$-property \underline{iff} $\bfp_{\!A}$ has the continuous\\ $(D,\partial,\theta)$-$\BB^1$-property for some $A \in [\kappa]^\kappa$.
\end{claim}

\begin{PROOF}{\ref{a28}}
The $\Leftarrow$ implication is obvious, so we concentrate on $\Rightarrow$. 
Let $\bfp$ be a $(\lambda,\kappa,\chi)$-$\BB^1$-parameter and $D$ be as above.
\begin{enumerate}
    \item [$\circledast_1$] Toward contradiction, assume that $\bfp_{\!A}$ fails the continuous $(D,\partial,\theta)$-$\BB^1$-property for all $A \in [\kappa]^\kappa$.
\sn
    \item [$\circledast_2$] 
    \begin{enumerate}
        \item So for $A \in [\kappa]^\kappa$, let 
$$
        \olsi F_{\!A}^1 = \LL F_{\delta,i}^{A,1} : \delta \in S,\ i < \kappa\RR
$$
        be
        a continuous $(\bfp_{\!A},\partial,\theta)$-coloring witnessing this failure. 
        
        (I.e.\ there is no $\bfp_{\!A}$-$D$-$\olsi F_{\!A}$-$\BB^1$-sequence $\bar c \in {}^{S\times\kappa}\theta$.)
\sn
        \item Now $F_{\delta,i}^A : {}^{C_{\delta,i}}\partial \to \theta$ is defined as follows:
        \begin{enumerate}
            \item If $i \in A$ then $F_{\delta,i}^A \defeq F_{\delta,\otp(i \cap A)}^{A,1}$.
\sn
            \item If $i \in \kappa \setminus A$ then $F_{\delta,i}^A$ is the constantly zero.
        \end{enumerate}
    \end{enumerate}
\end{enumerate}
Naturally, we choose
\begin{enumerate}
    \item [$\circledast_3$] 
    \begin{enumerate}
        \item $\cd : {}^{([\kappa]^\kappa)}\theta \to \theta$ and $\cd^* : {}^{([\kappa]^\kappa)}\partial \to \partial$, both bijections.
\sn
        \item For $B \in [\kappa]^\kappa$, let $\cd_B : \theta \to \theta$ be defined so that the following diagram commutes:
        \begin{figure}[H]
        \tikz{
            \node (A) at (-1,1) {${}^{([\kappa]^\kappa)}\theta$} ;
            \node (B) at (1,1) {$\theta$} ;
            \node (C) at (1,-1) {$\theta$} ;
            
            \draw (A) edge["$\cd$",->] (B);
            \draw (B) edge["$\cd_B$",->] (C);
            \draw (A) edge["$\pi_B$",->] (C);
        }
        \end{figure} 
        where $\pi_B$ is the function which sends $\LL \zeta_A : A \in [\kappa]^\kappa\RR \mapsto \zeta_B$.
\sn
        \item $\cd_B^* : \partial \to \partial$ will be defined analogously.
    \end{enumerate}
\end{enumerate}
Next, 
\begin{enumerate}
    \item [$\circledast_4$] Choose $\olsi F = \LL F_{\delta,i} : \delta \in S,\ i < \kappa\RR$ as follows:

    $\dom(F_{\delta,i}) \defeq {}^{C_{\delta,i}}\partial$, and for $\eta$ in the domain we define
    $$
    F_{\delta,i}(\eta) \defeq \cd \big(\LL F_{\delta,i}^A(\eta) : A \in [\kappa]^\kappa\RR \big).
    $$
\end{enumerate}
By our assumption,
\begin{enumerate}
    \item [$\circledast_5$] There exists a $\bfp$-$D$-$\olsi F$-$\BB^1$-sequence $\bar c = \LL c_{\delta,i} : \delta \in S,\ i < \kappa\RR\in {}^{S\times\kappa}\theta$.
\end{enumerate}
Next,
\begin{enumerate}
    \item [$\circledast_6$] For every $A \in [\kappa]^\kappa$, we choose $\bar c_A \defeq \big\LL \cd_A(c_{\delta,i}) : \delta \in S,\ i < \kappa \big\RR$.
\end{enumerate}
If $\bar c_A$ is a $\bfp_{\!A}$-$D$-$\olsi F_{\!A}$-$\BB^1$-sequence for some $A$, then we get our contradiction. 

Therefore, assume:
\begin{enumerate}
    \item [$\circledast_7$] For each $A \in [\kappa]^\kappa$ there exist $\eta_A \in {}^\lambda\partial$ and $E_A \in D$ such that
    $$
    (\forall\delta \in S \cap E_A) (\exists i < \kappa) \big[ F_{\delta,h_A^{\supminus1}\!(i)}^{A,1}(\eta_A \rest C_{\delta,h_A^{\supminus1}\!(i)}) \neq c_{\delta,h_A^{\supminus1}\!(i)}^A \big].
    $$
    (Equivalently, $(\forall\delta \in S \cap E_A) (\exists i \in A) \big[ F_{\delta,i}^A(\eta_A \rest C_{\delta,i}) \neq c_{\delta,i}^A \big].$)
\end{enumerate}
Now,
\begin{enumerate}
    \item [$\circledast_8$] $E \defeq \bigcap\limits_{A\in[\kappa]^\kappa} E_A \in D$.
\end{enumerate}
[Why? Because we assumed $D$ is $(2^\kappa)^+$-complete.]

\sn
Next,
\begin{enumerate}
    \item [$\circledast_9$] Define $\eta \in {}^{([\kappa]^\kappa)}\partial$ as the function 
    $$
    \alpha \mapsto \cd^*\!\big(\LL \eta_A(\alpha) : A \in [\kappa]^\kappa\RR\big).
    $$
\end{enumerate}
Now we can finish as in the proof of \ref{a8}.
\end{PROOF}

\mn
\begin{conclusion}\label{a31}
Assume clause $(A)$ of Theorem \emph{\ref{a8}}. Also suppose 
$\clD_\bfp = [\kappa]^\kappa$, $2^\kappa \leq \mu$, and 
$\theta^{2^\kappa} = \theta$.

\sn
For some $A\in[\kappa]^\kappa$, $\bfp_{\!A}$ has the $(D,2^\mu,\theta)$-$\BB^1$-property.
\end{conclusion}

\begin{PROOF}{\ref{a31}}
By \ref{a8} we know 
\begin{enumerate}
    \item [$(*)_1$] $\bfp$ has the continuous $(D,2^\mu,\theta)$-$\BB^1$-property.
\end{enumerate}

Let $\partial \defeq 2^\mu$.
We would like to apply \ref{a28}, so let us check its assumptions.
First, $\bfp$ is a $(\lambda,\kappa,\chi)$-$\BB^1$-parameter by \ref{a8}(A)(d), one of our assumptions.

We also assumed $\clD_\bfp = [\kappa]^\kappa$ and 
$\theta^{2^\kappa} = \theta$, so we don't have to worry about those.

`$D$ is a $\mu^+$-complete filter' was assumed in \ref{a8}(A)(b) and we added $2^\kappa \leq \mu$, so $D$ is also $(2^\kappa)^+$-complete.

Lastly, $\partial^{2^\kappa} = \partial$ holds because we defined 
$\partial \defeq 2^\mu$; hence $(2^\mu)^\mu = 2^\mu$, and again 
$2^\kappa \leq \mu$.

Therefore the conclusion in \ref{a28} holds, giving us our desired conclusion.
\end{PROOF}

\mn
\begin{conclusion}\label{a37}
$1)$ We can add the following to the conclusion of \emph{\ref{a8}}.

If $\lambda_* = \cf(\lambda_*) \geq \lambda$, \underline{then} there exists a $\bfp^*$ such that
\begin{enumerate}[$(a)$]
    \item $\bfp^*$ is a good $(\lambda_*^+,\kappa,\chi)$-$\BB_*$-parameter.
\sn
    \item $\bfp^*$ has the continuous $\lambda$-uniform $(D,2^\mu,\theta)$-$\BB^0$-property.
\end{enumerate}

\sn
$2)$ If $\bfp$ is a $(\lambda,\kappa,\chi)$-$\BB^0$-parameter with the 
$(D,2^\mu,\theta)$-$\BB^0$-property, \underline{then} $(\bfp,\{\kappa\})$ is a $(\lambda,\kappa,\chi)$-$\BB^1$-parameter with the $(D,2^\mu,\theta)$-$\BB^1$-property.
\end{conclusion}

\begin{PROOF}{\ref{a37}}
1) By \cite[\S4]{Sh:351}, as in \cite[\S2]{Sh:775}.

\sn
2) Easy, by the definitions.
\end{PROOF}

\mn
\begin{remark}\label{a40}
We can say more in \ref{a37}, replacing $\lambda_*^+$ by $\lambda'$ weakly inaccessible or successor of singular: see \cite{Sh:775}.
\end{remark}

\bn
\centerline{*\qquad*\qquad*}

\bigskip
We are now ready to prove Theorem \ref{y2} as promised in the introduction.
\begin{quotation}
    The reader would be well-advised to keep the statements of Theorem \ref{y2} on page \pageref{y2}, Definition \ref{a2}(1) on page \pageref{a2}, and Lemma \ref{a8} on page \pageref{a8} close at hand while reading this proof.
\end{quotation}
\begin{PROOF}[\textsc{Proof of \ref{y2}:}]{\ref{y2}}\label{a43}




Let us define a $\BB^1$-parameter $\bfp$ (as in Definition \ref{a2}) as follows:
\begin{enumerate}[$\bullet_1$]
    \item $(\lambda_\bfp,\kappa_\bfp) \defeq (\lambda,\kappa)$ from \ref{y2}(A)(a)-(b) (so $\lambda > \kappa$ are regular, and so \ref{a2}(1)(A)(a) holds).
\sn
    \item $\chi \defeq \min\{\sigma : 2^\sigma > \mu\}$ (so $\chi \leq \mu$).
\sn
    \item $S_\bfp \defeq S$, the stationary subset of $S_\kappa^\lambda$ from \ref{y2}(A)(b)\\ (so \ref{a2}(1)(A)(b) holds).
\sn
    \item $C_{\bfp,\delta}' \defeq C_\delta'$ from \ref{y2}(A)(c). They satisfy \ref{a2}(1)(C)(a), and so we let $\bar\beta_\delta =  \LL \beta_{\delta,i} : i < \kappa\RR$ list its elements as in \ref{a2}(1)(C)(b). 
\sn
    \item Let $C_{\bfp,\delta} \defeq \bigcup\limits_{i<\kappa} C_{\beta_{\delta,i}}^\bullet$, where $\LL C_\beta^\bullet : \beta < \lambda\RR$ was given in \ref{y2}(A)(d).
\sn
    \item If $\beta \in C_\delta' \wedge \otp(C_\delta' \cap \beta) = i$ \underline{then} we define $C_{\delta,i} \defeq C_\beta^\bullet$.\footnote{
        Note that this is well-defined, as the sequence 
        $\LL \otp(\beta \cap C) : \beta \in C\RR$ is strictly increasing for any set $C$ of ordinals.
    }
    \begin{enumerate}
        \item Clearly each $C_{\delta,i} \subseteq C_{\bfp,\delta} \cap \beta_{\delta,i}$, giving us \ref{a2}(1)(D)(b).
\sn
        \item Also, $|C_{\delta,i}| < \chi$ as $2^{|C_{\delta,i}|} \leq 2^\mu$ by \ref{y2}(A)(d), hence \ref{a2}(1)(D)(c) holds.
    \end{enumerate}
\sn
    \item $\clD_\bfp \defeq [\kappa]^\kappa$.
\end{enumerate}

Now,
\begin{enumerate}
    \item [$\boxplus_1$] $\bfp$ is indeed a $\BB^1$-parameter.
\end{enumerate}
[Why? The only demand we have not already checked off our list is clause \ref{a2}(1)(E). For a given $\alpha < \lambda$, the set defined there is {a singleton} $\{C_{\beta_{\delta,i}}^\bullet\}$, and so this easily holds.]

\mn
Let $D$ be the club filter on $\lambda$.
\begin{enumerate}
    \item [$\boxplus_2$] All the assumptions in Lemma \ref{a8} have been satisfied.
\end{enumerate}
Why? \ref{a8}(A)(a) holds by \ref{y2}(A). (We defined $\lambda \defeq \cf(2^\mu)$ in \ref{y2}(A)(b).)

Clause (A)(b) holds by our choice of $D$ above.

For (A)(c), 
$\chi \leq \mu$ by $\bullet_2$. 
Recalling $\theta = \theta^{2^\kappa}$ and $\theta < \mu$, clearly $\kappa < \chi$. As $\lambda \defeq \cf(2^\mu)$ ($> \mu$), we have $\lambda \geq \chi$. 

Clause (A)(d) demands that $\bfp$ is a $\BB^1$-parameter, which we have just proved.

For (A)(e), note that $\theta \leq \mu$ by \ref{y2}(A)(a). Let $\sigma < \chi$; necessarily, $2^\sigma \leq \mu$ by our choice of $\chi$, so as $\theta < \mu$ we have $\theta^\sigma \leq 2^\mu$. Therefore $\theta^{<\chi} \leq 2^\mu$ as required.

Lastly, we have to prove (A)(f):
$$
\Sep_3(\lambda;\mu,\mu,\partial,\theta,\kappa,\clD).
$$
[Why is this true? We shall prove it using \ref{x47}, 
recalling $\clD \defeq [\kappa]^\kappa$. Now we need to verify that demands \ref{x47}(A)-(D) hold. Clauses (A) and (D) were among our assumptions, and 
`$\bfU_\clD(\mu) = \mu$' holds because this is equivalent to 
$\bfU_\kappa(\mu) = \kappa$ (as $\clD \defeq [\kappa]^\kappa$), which holds because 
$\alpha < \mu \Rightarrow |\alpha|^\kappa \leq \mu$.

The last of the four clauses is \ref{x47}(B); if we have this, then we will be able to conclude $\Sep_3(\lambda;\mu,\mu,\partial,\theta,\kappa,\clD)$.

Flipping back to page \pageref{x47}, it says ``The triple $(\mu,2^\mu,\theta)$ satisfies at least one of the conditions in {\ref{x20}}.''
Of the options, \ref{x20}$(b)$ holds by \ref{y2}(A), and so we are done.]

\smallskip
We obtain the conclusion of \ref{x47}, giving us \ref{a8}(A)(f), 
and $\boxplus_2$ has been proven.

Now by the conclusion of Lemma \ref{a8}, $\bfp$ has the continuous $(D,2^\mu,\theta)$-$\BB^1$-property, and so by \ref{a28}, we know 
\begin{enumerate}
    \item [$\boxplus_3$] For some $A \in [\kappa]^\kappa$, $\bfp_{\!A}$ (as defined in \ref{a22}) has the continuous $(D,2^\mu,\theta)$-$\BB^1$-property.
\end{enumerate}

Now to finish the proof we need to provide a $\olsi C = \LL C_\delta : \delta \in S\RR$ satisfying \ref{y2}(B). Let us choose {$\LL C_\delta^A : \delta \in S\RR$, where}
$$
C_\delta^A \defeq \{\beta \in C_{\bfp,\delta} : \otp(C_{\bfp,\delta} \cap \beta) \in A\}.
$$ 
Now check.
\end{PROOF}

\mn
\begin{PROOF}[\textsc{Proof of \ref{y5}}]{\ref{y5}}
Similar.
\end{PROOF}

\newpage
\section{The DBB Property}\label{S3}

{The following result} relies on \cite[2.2\subref{d.6}]{Sh:898}. 
\begin{theorem}\label{b2}
We have $`(A) \Rightarrow (B)\text{\emph{'}}$, where
\begin{enumerate}[$(A)$]
    \item 
    \begin{enumerate}[$(a)$]
        \item $\lambda \defeq \min\{\partial : 2^\partial > 2^\mu\}$ (so $\lambda > \mu$ is regular).
\sn
        \item Let $D$ be a $\mu^+$-complete filter on $\lambda$ extending the co-bounded filter.
\sn
        \item $\olsi C = \LL C_\gamma : \gamma < \lambda\RR$, where $C_\gamma \subseteq \mu$.
\sn
        \item $\theta \in [2,\mu]$ 
\sn
        \item $\Sep_1(\mu,\mu,\theta,\theta,\Upsilon)$ for some $\Upsilon < 2^\mu$ (or possibly $\Upsilon \defeq 2^\mu \in \Reg$ as before).
    \end{enumerate}
\sn
    \item If $\bfF_\gamma : {}^{C_\gamma}(2^\mu) \to \theta$ 
    for $\gamma < \lambda$ \underline{then} we can find a 
    $$
    \bar c = \LL c_\gamma : \gamma < \lambda\RR \in {}^\lambda\theta
    $$ 
    such that for any 
    $f : \mu \to 2^\mu$, for $D^+$-many $\gamma < \lambda$, we have 
    $$
    \bfF_\gamma( f \rest C_\gamma) = c_\gamma.
    $$
\end{enumerate}
\end{theorem}

\begin{PROOF}{\ref{b2}}
By \cite[2.2$(\beta)$\subref{d.6}]{Sh:898}.
\end{PROOF}

\mn
\begin{definition}\label{b5}
Suppose $\lambda = \cf(\lambda) > \mu \ge \kappa = \cf(\kappa)$ and $\mu_* \leq \mu^+$. 

\sn
1) We say that $\bfp$ is a $(\lambda,\lambda_*,\mu,\mu_*,\kappa)$-\emph{DBB-parameter}\footnote{
    DBB stands for \emph{Double Black Box.}
} 
\underline{when}:
\begin{enumerate}
    \item 
    \begin{enumerate}
        \item $\lambda \geq \lambda_* > \kappa$ are regular cardinals.
\sn
        \item $S \subseteq S_\kappa^\lambda$ is a stationary subset of $\lambda$.
    \end{enumerate}    
\sn    
    \item $\bfp$ consists of $\olsi C = \olsi C_0 = \LL C_\delta : \delta \in S\RR$ and $\olsi C_1 = \LL C_\gamma^\delta : \delta \in S,\ \gamma < \lambda_*\RR$ such that
    \begin{enumerate}
        \item $\LL C_\delta : \delta \in S\RR$ is as usual 
        (that is, $C_\delta \subseteq \delta = \sup (C_\delta)$ and $\otp(C_\delta) = \kappa$),
        but we add the demand
        $$
        \alpha \in C_\delta \Rightarrow \alpha > \mu \wedge \mu \divides \alpha.
        $$
        \item $C_\gamma^\delta \subseteq \bigcup\limits_{\alpha \in C_\delta} [\alpha,\alpha+\mu)$ such that 
        $$
        \big| C_\gamma^\delta \cap [\alpha,\alpha+\mu) \big| = 1
        $$
        for all $\delta \in S$, $\gamma < \lambda_*$, {and $\alpha \in C_\delta$}. (So $\otp(C_\gamma^\delta) = \kappa$.) 
\sn
        \item $\olsi C_1$ is a $\mu_*$-\emph{free sequence}.
        
        By this we mean: if $u \subseteq S  \times \lambda_*$ is of cardinality 
        $\leq \mu_*$, \underline{then} there exists some sequence 
        $\bar \beta = \LL \beta_\gamma^\delta : (\delta,\gamma) \in u\RR$ with
        $\beta_\gamma^\delta \in C_\gamma^\delta$ such that
        $$
        \LL C_\gamma^\delta \setminus \beta_\gamma^\delta : (\delta,\gamma) \in u\RR
        $$
        is a sequence of pairwise disjoint sets.
    \end{enumerate}
\end{enumerate}

\sn
2) We say that $\bfp$ has the $(\lambda,\lambda_*,\mu,\mu_*,\theta,\kappa)$-\emph{DBB-property} \underline{when} in addition to the above, 
\begin{enumerate}
    \item [(C)]
    If $\bfF_\gamma^\delta : {}^{C_\gamma^\delta}(2^\mu) \to \theta$ 
    for $\gamma < \lambda_*$ and $\delta \in S$, \underline{then} we can find sequences
    $$
    \bar c^\delta = \LL c_\gamma^\delta : \gamma < \lambda_*\RR \in {}^{\lambda_*}\theta
    $$ 
    such that for any $\delta \in S$ and 
    $f : \delta \to 2^\mu$, for some $\gamma < \lambda_*$, we have 
    $$
    \bfF_\gamma^\delta( f \rest C_\gamma^\delta) = c_\gamma^\delta.
    $$
\end{enumerate}

\sn
3) If we say $\bfp$ \emph{guesses clubs}, we mean $\olsi C_\bfp$ does.

If $\lambda_* \defeq \lambda$ then we may omit it. Similarly if {$\mu_* \defeq \mu^+$.}
\end{definition}

\mn
\begin{remark}
1) In Definition \ref{b5}, we can make the following changes:
\begin{enumerate}[$\bullet_1$]
    \item In clause (1)(B), we add
    \begin{enumerate}
        \item [(B)(d)] $\olsi D = \LL D_\delta : \delta \in S\RR$, with each $D_\delta$ a filter on $\lambda_*$.
    \end{enumerate}
\sn
    \item Then in clause (2)(C), we replace ``for some $\gamma < \lambda_*$" by `for $D_\delta$-many $\gamma < \lambda_*$.'
\end{enumerate}

\mn
2) Adopting this change, we would add an additional clause to Claim \ref{b8}(1)$\boxplus$:
\begin{enumerate}
    \item [$\boxplus(f)$]
    \begin{enumerate}[$\bullet_1$]
        \item $\olsi D = \LL D_\delta : \delta \in S\RR$, with each $D_\delta$ a filter on $\lambda_*$.
\sn
        \item $\lambda_\bullet \in D_\delta$
\sn
        \item $D_\delta \rest \lambda_\bullet$ is a $\mu^+$-complete filter on $\lambda_\bullet$ extending the co-bounded filter.
    \end{enumerate}
\end{enumerate}

\mn
3) The proof of \ref{b8} would not change.
\end{remark}

\bn
\begin{claim}\label{b8}
$1)$ If $\boxplus$ below holds, \underline{then} there exists a $\bfp$ with the $(\lambda,\lambda_*,\mu,\mu_*,\theta,\kappa)$-DBB-property.  
\begin{enumerate}
    \item [$\boxplus$]
    \begin{enumerate}[$(a)$]
        \item $\kappa = \cf(\mu) < \mu$
\sn
        \item $\lambda_\bullet \leq \lambda_* \leq \lambda$, where $\lambda$ is regular and $\lambda_\bullet \defeq \min\{\partial : 2^\partial > 2^\mu\}$.
\sn
        \item
        \begin{itemize}
            \item[$\bullet_1$] $\pp_{\!J_\kappa^\bd}(\mu) > \lambda$ and {$\mu_* \defeq \mu^+$,}
        \end{itemize}
        \emph{\underline{or}}
        \begin{itemize}
            \item[$\bullet_2$] $\mu_* \leq \mu$ and there exists a $\mu_*$-free subset of ${}^\kappa\mu$ of cardinality $\lambda$ (see \emph{\S4}).
        \end{itemize}
\sn
        \item $\theta \in [2,\mu]$
\sn
        \item $\Sep_1(\mu,\theta)$.
    \end{enumerate}
\end{enumerate}

\sn
\emph{1A)} If $(1)\boxplus$ holds and $S \subseteq\{\delta < \lambda : \cf(\delta) = \kappa \text{ and } \mu^2 \divides \delta\}$ is stationary in $\lambda$, \underline{then} there exists a $\bfp$ with the $(\lambda,\lambda_*,\mu,\mu_*,\theta,\kappa)$-DBB-property and $S_\bfp = S$.

\sn
\emph{1B)} If we assume $S \in \check I_\kappa[\lambda]$ then\footnote{
    Recall that such an $S$ exists because $\lambda$ and $\kappa$ are regular with $\lambda > \kappa^+$ (as $\lambda > \mu > \kappa$), and so we can apply \cite[\S1]{Sh:420}.
} 
we can add ``$\olsi C_\bfp$ is $(\lambda,\kappa)$-good" (see \emph{\ref{x5}(6)}) and ``$S_\bfp \defeq S$." 

\sn
\emph{1C)} Moreover, in \emph{(1B)} we could replace `$S_\bfp \defeq S$' by `$S_\bfp \defeq S \cap E$ for some club $E \subseteq \lambda$' and add `if $\beta \in C_{\bfp,\delta_1} \cap C_{\bfp,\delta_2}$ then 
$C_{\bfp,\delta_1} \cap \beta = C_{\bfp,\delta_2} \cap \beta$.'

\sn
\emph{1D)} In both \emph{(1A)} and \emph{(1B)}, we may add $``\olsi C_\bfp$ guesses clubs.\emph{"}

\sn
\emph{1E)} Note that if $\mu$ is (e.g.) strong limit singular, then $(1)\,{\boxplus}(c)$ holds even if $\cf(\mu) = \aleph_0$.

\sn
$2)$ In part $(1)$, we may replace clause $\boxplus\,(c)$ by
\begin{enumerate}
    \item [$(c)'$]
    \begin{enumerate}[$\bullet_1$]
        \item $\pp_J(\mu) > \lambda$ for some ideal $J \supseteq [\kappa]^{<\kappa}$.
\sn
        \item $\mu = \mu^{<\kappa}$
    \end{enumerate}
\end{enumerate}

\sn
\emph{2A)} We can adopt $\bullet_1$ above if we weaken clause \emph{\ref{b5}(1)(B)(c)} to $``(\mu_*,J)$-free.\emph{"}

\sn
$3)$ Alternatively,
\begin{enumerate}
    \item [$(c)''$]
    \begin{enumerate}[$\bullet_1$]
        \item As above.
\sn
        \item $2^{\mu^{<\kappa}} < 2^\mu$.
    \end{enumerate}
\end{enumerate}
\end{claim}

\sn
\begin{remark}\label{b9}
1) Concerning the Double Black Box property in Definition \ref{b5}(2), we may allow $\bfF_\gamma^\delta$ to have domain ${}^{\clC_\gamma^\delta}(2^\mu)$, where
$\clC_\gamma^\delta \defeq \bigcup\limits_{\beta \in C_\gamma^\delta} [\beta,\beta+\mu)$, and range $\subseteq \theta$.
(Alternatively, we may use $\clC_\gamma^\delta \defeq \big\{ \beta : (\exists \beta' \in C_\gamma^\delta)[ \beta + \mu = \beta' + \mu] \big\}.$)

\noindent
[Why? Because $\big|{}^\mu(2^\mu) \big| = 2^\mu$.]

\mn
2) Assume $\lambda \defeq 2^\mu \notin \Reg$. Then the proof still works for a weaker version of the DBB property, where `$\lambda$ is regular' (in \ref{b5}(1)(A)(a)) is replaced by `$\cf(\lambda) > \kappa$.' (Of course, we still demand that $\lambda_*$ and $\kappa$ are regular.)

\mn
3) By \S4, if $\mu$ is strong limit singular above then we can have $\mu_* \defeq \mu^+$.
\end{remark}

\sn
\begin{PROOF}{\ref{b8}}
1) First, choose a stationary $S \subseteq S_\kappa^\lambda$ such that
$$
\delta \in S \Rightarrow \mu^2 \divides \delta.
$$

Next, choose 
$\olsi C$ as in \ref{b5}(1)(B)(a); this is possible by our choice of $S$. 
Third, by $\boxplus$, choose a $\mu_*^+$-free sequence
$$
\LL\rho_\gamma : \gamma < \lambda\RR \subseteq {}^\kappa\!\mu.
$$ 
[Why can we do this? In subclause $\boxplus(c)\bullet_1$, this follows from \cite[Ch.II, \S3]{Sh:g}; if $\bullet_2$ holds then this is obvious.]

\medskip
Without loss of generality $\gamma < \lambda \Rightarrow \rho_\gamma(i) = i \mod \kappa$. 
Let $\LL \rho_{\delta,\gamma}^* : \delta \in S,\ \gamma < \lambda\RR$ list $\LL\rho_\gamma : \gamma < \lambda\RR$ without repetition: we can do this because $|S \times \lambda| = \lambda$.

Let $\LL \beta_i^\delta : i < \kappa\RR$ list $C_\delta$ in increasing order, and let 
\begin{itemize}
    \item $\rho_{\delta,\gamma} \defeq \LL \beta_i^\delta + \rho_{\delta,\gamma}^*(i) : i < \kappa\RR$
\sn
    \item $C_\gamma^\delta \defeq \rang(\rho_{\delta,\gamma})$.
\end{itemize}
Let $D$ be the club filter on $\lambda$. 
So
$$
\bfp \defeq (\lambda,\kappa,S,D,\LL C_\delta : \delta \in S\RR, \LL C_\gamma^\delta : \delta \in S,\ \gamma < \lambda\RR)
$$
is well-defined. 

Next we have to check that $\bfp$ is indeed a $(\lambda,\lambda_*,\mu,\mu_*,\theta,\kappa)$-DBB-parameter: that is, all clauses of \ref{b5}(1).

First, the demands on the cardinals in the beginning of the definition hold, as does clause (A).

\sn
\textbf{Clause \ref{b5}(B)(a):} Holds by the choice of $\LL C_\delta : \delta \in S\RR$.

\sn
\textbf{Clause (B)(b):} Holds by our choice of the $C_\gamma^\delta$-s.

\sn
\textbf{Clause (B)(c):} 


Let $u \in [S\times \lambda]^{<\mu_*}$. By the choice of 
$\LL \rho_\gamma : \gamma < \lambda\RR$, we can find a function 
$h : u \to \kappa$ such that
\begin{enumerate}
    \item [$(*)_1$] If $(\delta_1,\gamma_1) \neq (\delta_2,\gamma_2)$ are from $u$ and $i \geq \max\!\big( h(\delta_1,\gamma_1),h(\delta_2,\gamma_2)\big)$, \underline{then} 
    $\rho_{\delta_1,\gamma_1}^*(i) \neq \rho_{\delta_2,\gamma_2}^*(i)$.
\end{enumerate}
Hence
\begin{enumerate}
    \item [$(*)_2$] If $(\delta_1,\gamma_1) \neq (\delta_2,\gamma_2)$ are from $u$ and $i_\ell \in \big[h(\delta_\ell,\gamma_\ell),\kappa\big)$ for $\ell = 1,2$, \underline{then} 
    $$
    \rho_{\delta_1,\gamma_1}^*\!(i_1) \neq \rho_{\delta_2,\gamma_2}^*(i_2).
    $$
\end{enumerate}
[Why? If $i_1 \neq i_2$ then recall $\gamma < \lambda \Rightarrow \rho_\gamma(i) \equiv i \mod \kappa$. The $i_1 = i_2$ case is just $(*)_1$.]

So clause (B)(d) does indeed hold.
Together we have proved that $\bfp$ is a\\ $(\lambda,\lambda_*,\mu,\mu_*,\theta,\kappa)$-DBB-parameter.

Lastly, why does it have the $(\lambda,\lambda_*,\mu,\mu_*,\theta,\kappa)$-DBB-property? 

Let $\LL \bfF_\gamma^\delta : \gamma < \lambda_*,\ \delta \in S\RR$ be given, as in \ref{b5}(2)(C). To finish we need to produce $\bar c_\delta = \LL c_\gamma^\delta : \gamma < \lambda_*\RR$ as required there.

Fix $\delta \in S$, and let 
$h_\delta : \mu \to \delta$ be such that
$$
\eps < \mu \wedge \eps \equiv i\ \mathrm{mod}\ \kappa \Rightarrow h_\delta(\eps) \defeq \beta_{\delta,i} + \eps.
$$
As $\LL \beta_{\delta,i} : i < \kappa\RR$ is increasing and $\mu \divides \beta_{\delta,i}$ for all $i < \kappa$, we know 
$h_\delta$ is well-defined and one-to-one.

Let $\olsi C_\delta' = \LL C_{\delta,\gamma}' : \gamma < \lambda_\bullet\RR$ be defined as
$$
C_{\delta,\gamma}' \defeq \big\{\eps < \mu : h_\delta(\eps)\in C_\gamma^\delta \big\}.
$$

For $\gamma < \lambda_\bullet$, define $F_{\delta,\gamma} : {}^{C_{\delta,\gamma}'}(2^\mu) \to \theta$ as the function
$$
f \mapsto \bfF_\gamma^\delta (f \circ h_\delta).
$$
By \ref{b2} applied to $\olsi C_\delta'$ and 
$\LL F_{\delta,\gamma} : \gamma < \lambda_\bullet\RR$, we get a sequence of ordinals 
$\LL c_\delta^\gamma : \gamma < {\lambda_\bullet}\RR$ as guaranteed there.

As $\lambda_*$ may be any cardinal in the interval $[\lambda_\bullet,\lambda]$, we need to pad out this sequence with extra terms. Simply let $c_\delta^\gamma \defeq 0$ {(and $F_{\delta,\gamma}$ be identically zero)} for $\gamma \in [\lambda_\bullet,\lambda_*)$, {and we reader may check that the resulting sequence is as desired.}

\mn
1A) Similarly.

\sn
1B-C) By the definition of $\check I_\kappa[\lambda]$ and Claim \ref{y23}.

\sn
1D) Use \cite[Ch.III,\,\S1]{Sh:g} and \ref{y23}(3).

\sn
1E) By \S4.

\sn
2) Similarly, but when choosing $\bar\rho = \LL\rho_\gamma : \gamma < \lambda\RR$ we only require that it is $(\mu^+,J)$-free.

Then we let $\cd : {}^{\kappa>}\!\lambda \to \lambda$ be a bijection.

\sn
3) Similarly as well.
\end{PROOF}

\mn
\begin{discussion}\label{b11}
Let $\mu$ be strong limit singular of cofinality $\kappa < \mu$, and $\lambda = \lambda_* \defeq \min\{\partial : 2^\partial > 2^\mu\}$. 
\begin{enumerate}
    \item 
    \begin{enumerate}
        \item If $\lambda < 2^\mu$ and $\kappa > \aleph_0$, then \ref{b8}\,$\boxplus$ holds (see \cite{Sh:g}).
\sn
        \item What about $\lambda < 2^\mu$ and $\kappa \defeq \aleph_0$? Still, we knew \ref{b8}\,$\boxplus$ held in many cases
        (e.g.~for a club of $\mu < \mu_*$, where $\mu_* \defeq \beth_\delta > \kappa \defeq \cf(\mu_*) > \aleph_0$). In \S4 we shall prove that it \emph{always} holds.
\sn
        \item \ref{b8} would seem to be helpful for constructing (e.g.) $\mu^+$-free Abelian groups.
    \end{enumerate}
\sn
    \item But what about the $\lambda = \lambda_* = 2^\mu$ case? In this case we have $\lambda = \lambda^{<\lambda}$, a condition which is again helpful in constructions. Can we construct an entangled linear order of cardinality $\lambda^+$? Recall that by \cite{Sh:460} or \cite{Sh:829} we have $(D\ell)_\Lambda^{\!*}$. Can we use several pairwise disjoint subsets of $\lambda$?

    Alternatively, find a subset of ${}^\lambda\theta$ for some regular 
    $\theta$ (e.g.~$\cf(2^{\aleph_0})$)?
\sn
    \item Again, if $\lambda = 2^\mu$ then we may try to use 
    $$
    \gd \defeq \big\{ \theta \in \mu \cap \Reg : \big( \exists \mu' \in (\mu,\lambda) \big) \big[\cf(\mu') = \theta \wedge \pp_{\theta\text{-comp}}(\mu') =^+ \lambda \big] \big\}
    $$
    as in \cite{Sh:898} whenever $\gd \defeq \{\kappa\}$ does not work. The new proof is as in \cite{Sh:1028}, using \cite{Sh:1008}.
\sn
    \item However, we {can} use $\BB_k$ in clause (C). We consider 
    $\mu_0 < \ldots < \mu_{3n}$ as above (i.e.~all strong limit of 
    cofinality $\kappa < \mu_0$). For each $\ell$ we choose $\bfp_\ell$ as in \ref{b8}, except that their free-ness (in the sense of \cite{Sh:1028}) is such that their {``product"} is 
    $\aleph_{n\cdot\kappa^+}$-free, and they have a Black Box as there.
\end{enumerate}
\end{discussion}

\mn
\begin{definition}\label{b16}
1) For $\Lambda_* \subseteq \Lambda_\bullet \subseteq {}^\kappa\!\mu$, we say that $\Lambda_\bullet$ is \emph{$(\theta_2,\theta_1)$-free over}\footnote{
    We may omit $\lambda_*$ if it is empty.
} 
$\Lambda_*$ \underline{when} $\theta_2 \geq \theta_1$ and for every $\Lambda \subseteq \Lambda_\bullet \setminus \Lambda_*$ of cardinality $< \theta_2$ there is a \emph{witness} $(\bar\Lambda,h)$. By this we mean:
\begin{enumerate}
    \item $\bar\Lambda = \LL \Lambda_\gamma : \gamma < \gamma_*\RR$ is a partition of $\Lambda$ into $\gamma_*$-many sets, each of cardinality $< \theta_1$ 
    (so $\gamma_*$ is an ordinal $< \theta_2$).
\sn
    \item $h : \Lambda \to \kappa$.
\sn
    \item If $\gamma < \gamma_*$, $\eta \in \Lambda_\gamma$, and $i \in \big[ h(\eta),\kappa \big)$, \underline{then} 
    $$
    \eta(i) \notin \big\{\rho(j) : j < \kappa,\ \rho \in \textstyle\bigcup\limits_{\beta<\gamma} \Lambda_\beta \cup \Lambda_* \big\}.
    $$
\end{enumerate}

\sn
2) For $\Omega \subseteq \big\{ (\theta_2,\theta_1) \in \Card \times \Card : \theta_2 \geq \theta_1\big\}$, we say $\Lambda_\bullet$ is $\Omega$-free over $\Lambda_*$ \underline{when} it is $(\theta_2,\theta_1)$-free over $\Lambda_*$ for every $(\theta_2,\theta_1) \in \Omega$.

\mn
3) Suppose $\lambda = \cf(\lambda) > \mu \ge \kappa = \cf(\kappa)$ and $\Omega$ is as in part (2).

We say that $\bfp$ is a $(\lambda,\lambda_*,\mu,\Omega,\kappa)$-\emph{DBB-parameter} \underline{when} clauses (A) and (B)(a)-(b) of Definition \ref{b5} hold, and
\begin{enumerate}
    \item [(B)(c)$'$] $\olsi C_1$ is $\Omega$-free.
\end{enumerate}
\end{definition}

\mn
\begin{observation}\label{b13}
Assume (for transparency) that $\Lambda_\bullet \subseteq {}^{\kappa>}\!\mu$ is tree-like. (That is, $\eta \neq \nu \in \Lambda_\bullet \wedge \eta(i) = \nu(j)\ \Rightarrow\ i = j \wedge \eta \rest i = \nu \rest i$.)
 
If $\Lambda_\bullet$ is of cardinality $< \theta$ and $(\theta,\kappa^+)$-free over $\varnothing$, \underline{then} $\Lambda_\bullet$ is free.
\end{observation}

\begin{PROOF}{\ref{b13}}
See \cite[\S1]{Sh:1028}.
\end{PROOF}

\mn
\begin{claim}\label{b20}
$1)$ If $\boxplus$ holds \underline{then} there exists a $\bfp$ with the $(\lambda,\lambda,\mu,\Omega,\theta,\kappa)$-DBB-property, where
\begin{enumerate}
    \item [$\boxplus$]
    \begin{enumerate}[$(a)$]
        \item $\kappa = \cf(\mu) < \mu$
\sn
        \item $\lambda = \lambda^{<\lambda} = 2^\mu$
\sn
        \item $\pp_{\!J_\kappa^\bd}^+(\mu) > \lambda$
\sn
        \item $\theta \in [2,\mu]$
\sn
        \item $\Sep(\mu,\theta,\Upsilon)$ for some $\Upsilon \leq \mu$.
\sn
        \item $\Omega \defeq \{(\kappa^{+\kappa},\kappa^{+4})\}$
    \end{enumerate}
\end{enumerate}

\sn
$2)$ Like part $(1)$, but replacing clause $\boxplus(f)$ by
\begin{enumerate}
    \item [$(f)'$] $\Omega \defeq \big\{(\theta^{+\kappa},\theta^{+4}) : \theta \in [\kappa,\mu)\big\}$
\end{enumerate}

\sn
$3)$ In parts $(1)$ and $(2)$, we my replace clause $\boxplus(c)$ by 
\begin{enumerate}
    \item [$(c)'$]
    \begin{enumerate}[$\bullet_1$]
        \item $\pp_J(\mu) \geq \lambda$ for some ideal $J \supseteq [\kappa]^{<\kappa}$.
\sn
        \item $\mu = \mu^{<\kappa}$
    \end{enumerate}
\end{enumerate}
as in \emph{\ref{b8}(2)}.

\sn
$4)$ If $S$ is a stationary subset of $\{\delta < \lambda : \cf(\delta) = \kappa\}$ then we can demand $S_\bfp \defeq S$, and we can add ``$\olsi C_\bfp = \LL C_{\bfp,\delta} : \delta \in S\RR$ guesses clubs." If $S \in \check I_\kappa[\lambda]$ then we can add ``$\olsi C_\bfp$ is $(\lambda,\kappa)$-good."

\sn
$5)$ In part $(2)$, we can replace $\mu = \mu^{<\kappa}$ by $2^{\mu^{<\kappa}} < 2^\lambda$.
\end{claim}

\begin{remark}\label{b23}
The $\lambda = \lambda^{<\lambda}$ is not necessary; just otherwise {\ref{b8}} gives us more.
\end{remark}

\begin{PROOF}{\ref{b20}}
1) Like the proof of \ref{b8}, but 
in the choice of $\bar\rho$ (at the beginning of the proof) we replace `$\mu^+$-free' by `$\Omega$-free.' 

\sn
[Why is this possible? Use $\theta$ in the beginning of the proof of \cite[1.26\subref{a51}]{Sh:1028} (which relies on \cite[0.4-0.6\subref{y19,y22,y40}]{Sh:1008}).]

\sn
2) As above.

\sn
3) Similarly to the proof of \ref{a8}(2).

\sn
4-5) Clear.
\end{PROOF}

\newpage
\section{When do quite free subsets $\Lambda \subseteq {}^\kappa\!\mu$ exist?}\label{S4}

We know that if $\kappa \defeq \cf(\mu) < \mu < \lambda < \pp_J(\mu)$, where $J$ is an ideal on $\kappa$, \underline{then} there exists a $<_J$-increasing $(\mu^+,J)$-free sequence $\bar f \in {}^\lambda({}^\kappa\!\mu)$. This implies that if $\mu$ is strong limit singular of uncountable cofinality and $\lambda \in (\mu,2^\mu)$, \underline{then} there is a $\mu^+$-free subset $\Lambda \subseteq {}^\kappa\!\mu$ of cardinality $\lambda$. This also holds for many cardinals $\mu$ with $\cf(\mu) = \aleph_0$ (see \cite{Sh:g}).
This has applications for (e.g.) the existence of $\mu^+$-free Abelian groups with trivial dual (The TDC$_\mu$).

We intend to prove this for all $\mu$-s.

This is an example of the thesis ``assuming negations of $\GCH$ may help in proving interesting results." Above, we considered singular cardinals. 
What about regulars? If $2^\theta = 2^{<\mu} < 2^\mu$ 
(hence $\theta < \mu$)
then by an old result with Devlin \cite{Sh:65} we have Weak Diamond on $\mu$. For results from this century, \cite{Sh:1035} with Chernikov showed that by pcf considerations, for every $\mu$ there exist $\mu_0 \defeq \mu < \mu_1 < \ldots < \mu_n \defeq 2^\mu$ for $n \leq 6$ such that 
$\trp_{\kappa_\ell}^+(\mu_\ell) > \mu_{\ell+1}$ for $\ell < n$ (where $\kappa_\ell \defeq \cf(\mu_\ell) < \mu_\ell$). (See \ref{d3} below.)

Can we get freeness results? Here we see if we can get ``for every $\ell < n$ there exists a quite free subset $\Lambda_\ell \subseteq \lim\limits_{\kappa_\ell}\clT$ of cardinality $\mu_\ell$."

\bigskip
A characteristic neat result {(referenced in \S0)} is as follows.
\begin{theorem}\label{d2}
If $\mu$ is strong limit singular, $\kappa \defeq \cf(\mu)$, and $\lambda \in (\mu,2^\mu)$, \underline{then} there is a $\mu^+$-free subset $\Lambda \subseteq {}^\kappa\!\mu$ of cardinality $\lambda$.
\end{theorem}

\mn
Recall:
\begin{definition}\label{d3}
1) For $\kappa = \cf(\kappa) \leq \mu$, let $\trp_\kappa^+(\mu)$ be the minimal 
$\lambda$ such that there is no sub-tree $\clT \subseteq {}^{\kappa>}\!\mu$ of cardinality $\mu$ with $\geq\lambda$-many $\kappa$-branches.

($\trp$ stands for \emph{tree power}.)

\mn
2) For $J$ an ideal on $\kappa < \mu$ and $\Lambda \subseteq {}^\kappa\mu$, we say 
$\Lambda$ is $(\lambda,J)$-\emph{free} when for every 
$\Lambda' \in [\Lambda]^{<\lambda}$ there exists $f : \Lambda' \to J$ such that
$$
\eta \neq \nu \in \Lambda' \wedge i \in \kappa \setminus \big( f(\eta) \cup f(\nu) \big) \Rightarrow \eta(i) \neq \nu(i).
$$

\mn
3) For $J$ an ideal on $\kappa$ and $\bar A = \LL A_\eps : \eps < \kappa\RR$, let $T_J^+(\bar A)$ be the minimal 
$\lambda$ such that there is no\footnote{
    We add the $\{0\}$ just so we don't have to worry about $A_\eps \defeq \varnothing$.
}
$\Lambda \subseteq \prod\limits_{\eps < \kappa}(A_\eps \cup \{0\})$ of cardinality $\lambda$ such that 
$$
\eta \neq \nu \in \Lambda \Rightarrow \big\{i < \kappa : 0 \neq \eta(i) \neq \nu(i) \neq 0 \big\} \equiv \kappa \mod J.
$$

\mn
3A) Let $\bfT_J^+(\bar A)$ be the minimal 
$\lambda$ such that there is no $\Lambda$ satisfying $\boxplus_{\bar A,\Lambda}$ below.
\begin{enumerate}
    \item [$\boxplus_{\bar A,\Lambda}$]
    \begin{enumerate}
        \item $\Lambda \subseteq \bigcup\limits_{a \in J^+} \prod\limits_{i \in a} (A_i \cup \{0\})$
\sn
        \item $|\Lambda| \geq \lambda$
\sn
        \item $\eta \neq \nu \in \Lambda \Rightarrow \big\{i \in \dom(\eta) \cap \dom(\nu) : \eta(i) = \nu(i) \big\} \in J.$
    \end{enumerate}
\end{enumerate}

\mn
3B) Above, if $\bar A$ is the constant sequence $\LL B : \eps < \kappa\RR$, then we may write $T_J^+(B)$ and $\bfT_J(B)$.

\mn
3C) The default value of $J$ is 
$[\kappa]^{<\kappa}$. If we omit $J$, this is what we mean.

\mn
4) For $\chi \geq \lambda \geq \theta \geq \sigma$, $\cov(\chi,\lambda,\theta,\sigma)$ means that there exists a 
$\cP \subseteq [\lambda]^{<\theta}$ of cardinality $\leq\chi$ such that every 
$u \in [\lambda]^{<\theta}$ is contained in the union of $<\sigma$-many members of $\cP$.
\end{definition}

\bn
\begin{observation}\label{d5}
$1)$ Without loss of generality we may strengthen \emph{\ref{d3}(2)} to
$$
\eta \neq \nu \in \Lambda' \wedge \big[ i \in \kappa \setminus  f(\eta) \big] \wedge \big[ j \in \kappa \setminus  f(\nu) \big] \Rightarrow \eta(i) \neq \nu(j).
$$

\mn
$2)$ Similarly in \emph{\ref{d3}(3)}.

\mn
$3)$ If $\mu = \mu^{<\kappa} < \mu^\kappa$ then $\trp_\kappa(\mu) = \mu^\kappa$ (and even $\trp_\kappa^+(\mu) = (\mu^\kappa)^+$). So if $\mu$ is strong limit of cofinality $\kappa$ then $\trp_\kappa(\mu) = 2^\mu$.

\mn
$4)$ If $\chi \geq \lambda \geq \theta > \sigma$ and $\sigma$ is regular uncountable, then $\cov(\chi,\lambda,\theta,\sigma)$ holds \underline{iff}
$$
\chi \geq \sup\!\big\{ \pp_{\cf(\mu)\text{-}\mathrm{comp}}(\mu) : \mu \in [\theta,\lambda],\ \cf(\mu) \in [\sigma,\theta) \big\}.
$$

\mn
$5)$ Assume $\bar A = \LL A_\eps : \eps < \kappa\RR$ with $|A_\eps| \ge 2^\kappa$ and $J$ an ideal on $\kappa$ such that
$$
u \in J^+ \Rightarrow T_J^+(\bar A) = T_{J \rest u}^+ (\bar A \rest u).
$$
{Then} $\bfT_J^+(\bar A) = T_J^+(\bar A)$ \underline{or} $\bfT_J^+(\bar A) = \sigma$ and $T_J^+(\bar A) = \sigma^+$ for some $\sigma$ of cofinality 
$\leq 2^\kappa$.

\mn
$6)$ In part $(5)$, if $J \defeq [\kappa]^{<\kappa}$ then $\bfT_J^+(\bar A) = T_J^+(\bar A)$.

\mn
$7)$ In parts $(5)$ and $(6)$, we may replace $2^\kappa$ by $\cf(\clP(\kappa) \setminus J,\subseteq)$.
\end{observation}

\begin{PROOF}{\ref{d5}}
1-2) As $\mu \geq \kappa$, we can define the function $\eta' \in {}^\kappa\mu$ by
$\eta'(i) \defeq \pr(\eta(i),i)$ for $i < \kappa$ 
(where $\pr : \mu \times \mu \to \mu$ is some bijection), and then replace $\eta$ by $\eta'$.

\mn
3) Classical.

\mn
4) By \cite{Sh:g}.

\mn
5) Clearly $\bfT_J^+(\bar A) \geq T_J^+(\bar A)$. To get the `$\leq$' direction, assume 
\begin{enumerate}
    \item [$\circledast_1$] $\chi \in \big( 2^\kappa,\bfT_J^+(\bar A) \big) \text{ (or just } \big| \clP(\kappa) / J \big| < \chi < \bfT_J^+(\bar A)).$
\end{enumerate}
For $\alpha < \chi$, let $\eta_\alpha \in \prod\limits_{i \in u_\alpha}A_i$ (where $u_\alpha \in J^+$) witness $\circledast_1$.

For each $u \in J^+$, define $W_u \defeq \{\alpha < \chi : u_\alpha = u\}$. If for some $u$ the set $W_u$ has cardinality $\chi$, then clearly 
$T_{J \rest u}^+(\bar A \rest u) > \chi$ and we are done.

So assume
$$
u \in J^+ \Rightarrow |W_u| < \chi;
$$
recalling $\chi > \big| \clP(\kappa) / J \big|$, necessarily $\cf(\chi) \leq \big| \clP(\kappa) / J \big|$ and $\chi = \sup\!\big\{ |W_u| : u \in J^+\big\}$.

Now $\LL \eta_\alpha : \alpha \in W_u\RR$ essentially witnesses `$T_{J \rest u}^+(\bar A \rest u) \geq |{W_u}|$,' and we can easily finish.

(Note that if $\bfT_J^+(\bar A) > \chi^+$ then we could have used $\chi^+$ instead of $\chi$.)

\mn
$6)$ for $A \in J^+ = \clP(\kappa) \setminus J$, let 
$h_A : \kappa \to A$ be the inverse of the function 
$$
i \mapsto \otp(A \cap i).
$$
For each $A_i \in \bar A$ (that is, for each $i < \kappa$), let $\pr_i$ be a bijection between $(A_i \cup \{0\}) \times J^+$ and $A_i$.

Beginning as in the proof of part (3), for each $\alpha < \chi$ we define $\eta_\alpha' \in \prod\limits_{i<\kappa} A_i$ as follows.
\begin{enumerate}
    \item [$\circledast_\alpha$] For $\alpha < \chi$ and $i < \kappa$, 
    we let $\eta_\alpha'(i) \defeq \pr_i\big(h_{u_\alpha}\!(i),u_\alpha\big)$
\end{enumerate}

Now $\LL \eta_\alpha' : \alpha < \chi\RR$ witnesses `$\bfT_J^+(\bar A) > \chi^+$.'

\mn
$7)$ Left to the reader.
\end{PROOF}

\bigskip
Now comes the section's main results: \ref{d17}, \ref{d20}, \ref{d26}. (The reader may concentrate on the case 
`$\kappa_* \defeq \kappa$, $\mu_* \defeq \mu$' and on part (1) of \ref{d17}.)
\begin{claim}\label{d17}
$1)$ If clauses $(a)$-$(e)$ below hold, \underline{then} there is a $(\mu_*^+,J)$-free subset $\Lambda \subseteq {}^\kappa\mu$ of cardinality $\lambda$.
\begin{enumerate}[$(a)$]
    \item $\kappa_* \leq \kappa$ are regular, and $J$ is a $\kappa_*$-complete ideal on $\kappa$.
\sn
    \item $\kappa \defeq \cf(\mu) < \mu_* \leq \mu \leq \lambda \leq \chi$
\sn
    \item $\mu^\kappa \geq \chi^+$
\sn
    \item $\cov(\chi,\lambda,\mu_*,\kappa_*^+)$
\sn
    \item $\alpha < \mu_* \Rightarrow |\alpha|^\kappa < \mu$
\end{enumerate}

\mn
$2)$ In part $(1)$, we may replace clauses $(c)$ and $(e)$ by
\begin{enumerate}
    \item [$(c)'$] 
    $T_J^+\big(\mu \big) > \chi^+$
\sn
    \item [$(e)'$] $\alpha < \mu \Rightarrow \bfT_J^+(|\alpha|) \leq \chi^+$.
\end{enumerate}

\mn
$3)$ If we weaken \emph{`}$T_J^+\big(\mu \big) > \chi^+$\emph{'} in clause $(c)'$ to 
\emph{`}There exists a sequence of ordinals $\olsi\alpha$ with $\mu = \bigcup\limits_{j<\kappa} \alpha_j$, each $\alpha_i < \mu$, and $\bfT_J^+(\olsi\alpha) > \chi^+$,\emph{'}
then we may still conclude that there is a $\mu_*^+$-free subset 
of $\bigcup\limits_{u \in J^+}\prod\limits_{i \in u} \alpha_i$.
\end{claim}

\begin{PROOF}{\ref{d17}}
1) We shall prove that part (1) follows from part (2). To that end, we need to prove that assumptions (1)(a)-(e) imply (2)$(c)'$,$(e)'$.

Concerning clause $(c)'$: 
as $\cf(\mu) = \kappa$, there exists an increasing sequence\\ 
$\LL \mu_i : i < \kappa\RR$ with limit $\mu$. By clause (e), without loss of generality $\mu_i = (\mu_i)^\kappa > 2^\kappa$. By clause (1)(c) and basic cardinal arithmetic,
$$
\big| \textstyle\prod\limits_i \mu_i\big| = \big| \textstyle\sum\limits_i \mu_i\big|^\kappa = \mu^\kappa.
$$

Hence there is a sequence 
$\LL \eta_\alpha : \alpha < \chi^+\RR$ of members of $\prod\limits_{i<\kappa}\mu_i$ without repetition, and a one-to-one function 
$\cd_i : {}^i(\mu_i) \to \mu_i$ for each $i < \kappa$. 
For $\alpha < \lambda$, we define $\nu_\alpha \in {}^\kappa\mu$ by 
$$
i < \kappa \Rightarrow \nu_\alpha(i) \defeq \cd_i(\eta_\alpha \rest (i+1)).
$$
Easily, $\alpha < \beta < \chi^+ \Rightarrow \big| \big\{ i < \kappa : \nu_\alpha(i) = \nu_\beta(i) \big\} \big| < \kappa.$
So by the assumption on $J$, the sequence $\LL \nu_\alpha : \alpha < \chi^+\RR$ witnesses $T_J^+(\mu) > \chi^+$.

As for clause (2)$(e)'$, it follows immediately from (1)(e).

\mn
2) First,
\begin{enumerate}
    \item [$(*)_1$] Let $\cP \subseteq [\lambda]^{<\mu_*}$ be the covering guaranteed by Definition \ref{d3}(4), recalling that we assumed $\cov(\chi,\lambda,\mu_*,{\kappa_*^+})$ in clause (1)(d).\\ (So in particular, $|\cP| \leq \chi$.)
\sn
    \item [$(*)_2$] Let $\Lambda^{\!*} \subseteq {}^\kappa\mu$ witness 
    $T_J^+(\mu) > \chi^+$ (as assumed in clause (2)$(c)'$). Note this implies $|\Lambda^{\!*}| \geq \chi^+$.
\end{enumerate}

\sn
We shall try to choose $\eta_\alpha$ and $\bar a_\alpha$ by induction on 
$\alpha < \lambda$ such that
\begin{enumerate}
    \item [$(*)_3^\alpha$] 
    \begin{enumerate}
        \item $\eta_\alpha \in \Lambda^{\!*}$
\sn
        \item $\bar a_\alpha = \LL a_u^\alpha : u \in \cP\RR$
\sn
        \item $a_u^\alpha \in J$
\sn
        \item if $u \in \cP$, $\beta \in u \cap \alpha$, and $i \in \kappa \setminus \big( a_u^\alpha \cup a_u^\beta \big)$, \underline{then} $\eta_\alpha(i) \neq \eta_\beta(i)$.
    \end{enumerate}
\end{enumerate}

Assuming we have succeeded up to Stage $\alpha$, for each $u \in \cP$ let 
$$
S_{\alpha,u} \defeq \big\{ \eta_\beta(i) : \beta \in u \cap \alpha,\ {i \in \kappa \setminus a_u^\beta}\big\}.
$$
So $S_{\alpha,u} \in [\mu]^{<|u|+\kappa}$; and as $|u| + \kappa < \mu_*$, we may say $S_{\alpha,u} \in [\mu]^{<\mu_*}$. 

For $\eta \in \Lambda^{\!*}$, let 
$$
W_{\alpha,u,\eta} \defeq \eta^{\supminus1}(S_{\alpha,u})
$$
and
$$
\Lambda_{\alpha,u} \defeq \big\{\eta \in \Lambda^{\!*} : W_{\alpha,u,\eta} \notin J \big\}.
$$
Now
\begin{enumerate}
    \item [$\bullet_{3.1}$] $|\Lambda_{\alpha,u}| < \bfT_J^+(S_{\alpha,u})$.
\end{enumerate}
[Why? Because $\big| \{\eta \rest W_{\alpha,u,\eta} : \eta \in \Lambda_{\alpha,u}\} \big| < \bfT_J^+(S_{\alpha,u})$ by its definition.]

\sn
As $|S_{\alpha,u}| \leq |u| +\kappa$, we may conclude $|\Lambda_{\alpha,u}| < \bfT_J^+(|u| +\kappa)$, and therefore
\begin{enumerate}
    \item [$\bullet_{3.2}$] $|\Lambda_{\alpha,u}| \leq \chi$.
\end{enumerate}
[Why? By clause $(e)'$ of our assumptions.]

\sn
Therefore (recalling $(*)_1$)
\begin{enumerate}
    \item [$(*)_4$] $\Lambda_\alpha \defeq \bigcup\limits_{u \in \cP} \Lambda_{\alpha,u}$ has cardinality $\leq |\cP| + \chi = \chi < |\Lambda^{\!*}|$.
\end{enumerate}

\sn
So we can choose $\eta_\alpha \in \Lambda^{\!*} \setminus \Lambda_\alpha$, and now: 
\begin{enumerate}
    \item [$(*)_5$] For each $u \in \cP$ with $u \ni \alpha$, we can choose $a_u^\alpha$ as required in $(*)_3^\alpha$(d).
\end{enumerate}
[Why? Because $\eta_\alpha \notin \Lambda_{\alpha,u}$.]

\sn
Lastly, let $\Lambda \defeq \{\eta_\alpha : \alpha < \lambda\}$.
\begin{enumerate}
    \item [$(*)_6$] $\Lambda \in [\Lambda^{\!*}]^\lambda$ is $\mu^+$-free.
\end{enumerate}
Why? First, $\bigcup\cP = \lambda$ by the definition of $\cP$ in $(*)_1$.
Hence if $\alpha < \beta < \lambda$ then for some $u \in \cP$ we have $\beta \in u$. {But} $\eta_\beta \notin \Lambda_{\alpha,u}$ by our choices, and hence 
$\eta_\alpha \neq \eta_\beta$.
Moreover, $\eta_\alpha \neq_J \eta_\beta$ (that is, $\{i < \kappa :  \eta_\alpha(i) = \eta_\beta(i)\} \in J$). Hence $|\Lambda| = \lambda$.

Second, let $u^* \in [\lambda]^{<\mu_*}$. Then by the choice of $\cP$ (in $(*)_1$) there is a sequence 
$$
\LL u_i : i < \kappa_*\RR \subseteq \cP
$$ 
such that $u^* \subseteq \bigcup\limits_i u_i$. For $\alpha \in u^*$, let $i_\alpha = i(\alpha) \defeq \min\{i < \kappa_* : \alpha \in u_i\}$ and let $a_\alpha^* \defeq a_{u_{i(\alpha)}}^\alpha$. Let 
$$
W \defeq \big\{ (\alpha,\beta) \in u^* \times u^* :  \big( \exists j \in \kappa \setminus (a_\alpha^* \cup a_\beta^*) \big) \big[ \eta_\alpha(j) = \eta_\beta(j) \big] \big\}.
$$

Easily, $\alpha \in u^* \Rightarrow \big|\{\beta : (\alpha,\beta) \in W\} \big| \leq \kappa_*$. Hence (noting that $W$ is a symmetric relation on $u^*$) we can find an equivalence relation $E$ on $u$ such that $(\alpha,\beta) \in W \Rightarrow \alpha\ E\ \beta$ and every equivalence class is of cardinality $\leq\kappa_*$.

For $S$ an equivalence class of $E$, let $\LL \alpha_j^S : j < j_S \leq \kappa\RR$ list $S$ without repetition. For $\alpha \defeq \alpha_j^S \in S$, let 
$$
a_\alpha \defeq a_\alpha^* \cup \big\{i \in \kappa_* : (\exists\eps < j)\big[ \eta_{\alpha_\eps^S}(i) = \eta_\alpha(i) \big] \big\}.
$$
This is a member of $J$ because {$j < \kappa_*$} and $J$ is $\kappa_*$-complete (by assumption $(1)(a)$), and by our choice of $\Lambda$.

So $\LL a_\alpha : \alpha \in u\RR$ witnesses the freeness demand in $(*)_6$.

\mn
3) Similarly, but in $(*)_2$ we choose $\Lambda^{\!*}$ to witness clause $(c)''$.
\end{PROOF}

\bn
\begin{claim}\label{d20}
If clauses $(a)$-$(f)$ below hold, \underline{then} there is a $(\mu^+,J)$-free subset $\Lambda \subseteq {}^\kappa\mu$ of cardinality $\lambda$.
\begin{enumerate}[$(a)$]
    \item $J$ is a $\kappa$-complete ideal on $\kappa$.
\sn
    \item $\kappa \defeq \cf(\mu) < \mu < \mu_\bullet \leq \lambda \leq \chi$
\sn
    \item $\bfT_J^+(\mu) > \chi^+$
\sn
    \item $\kappa_\bullet \defeq \cf(\mu_\bullet) \in [\kappa^+,\mu)$ and 
    $$
    \partial \in (\mu,\mu_\bullet) \wedge \cf(\partial) \in [\kappa^+,\mu) \Rightarrow \pp_{\cf(\partial)}(\partial) < \mu_\bullet.
    $$
    \item $\lambda^+ < \pp_{J_\bullet}^+\!(\mu_\bullet)$, where $J_\bullet$ is a $\kappa^+$-complete ideal on $\kappa_\bullet$.\footnote{
        Alternatively, maybe $J_\bullet$ is just $\big(\kappa+\cf(J,\subseteq)\big)^+$\!-complete.
    }
\sn
    \item $\alpha < \mu \Rightarrow |\alpha|^{\kappa_\bullet} < \mu$.
\end{enumerate}
\end{claim}

\begin{PROOF}{\ref{d20}}
First,
\begin{enumerate}
    \item [$(*)_1$] There exists $\olsi\alpha^\bullet = 
    \LL\alpha_i^\bullet : i < \kappa\RR \in {}^\kappa\mu$ such that 
    $T_J^+(\olsi\alpha^\bullet) > \chi^+$.
\end{enumerate}
[Why? by assumption $(c)$ there exists $\Lambda \subseteq \bigcup\limits_{b \in J^+} {}^b\mu$ witnessing $\bfT_J^+(\mu) > \chi^+$.

Now let $\LL \mu_i : i < \kappa\RR$ be increasing with limit $\mu$; so for every $\eta \in \Lambda$ there is an increasing function $h_\eta : \kappa \to \kappa$ such that $i \in \dom(\eta) \Rightarrow \eta(i) < \mu_{h_\eta(i)}$. As $\big| \{h_\eta : \eta \in \Lambda\} \big| \leq 2^\kappa \leq \chi$, clearly for some $h \in {}^\kappa\kappa$ the set $\Lambda_h \defeq \{\eta \in \Lambda : h_\eta = h\}$ has cardinality $\geq \chi^+$. 

So let $\alpha_i^\bullet \defeq \mu_{h(i)}$ witness $(*)_1$.

\sn
Second,
\begin{enumerate}
    \item [$(*)_2$] We can apply \ref{d17}(2) for each $\lambda' \in (\mu,\mu_\bullet)$, with $\kappa,\kappa,J,\mu,\mu,\lambda',\lambda'$ here standing in for $\kappa_*,\kappa,J,\mu_*,\mu,\lambda,\chi$ there.
\end{enumerate}
We have to check that all the assumptions hold.

\mn
\textbf{Clause (a):} Holds by clause $(a)$ of our assumptions.

\mn
\textbf{Clause (b):} Holds by clause $(b)$ of our assumptions, recalling $\mu < \lambda' < \mu_\bullet$.

\mn
\textbf{Clause (c)$'$:} Holds by $(*)_1$.

\mn
\textbf{Clause (d):} Holds by the second phrase in assumption $(d)$ and \ref{d5}(4).

\mn
\textbf{Clause (e),(e)$'$:} Directly implied by assumption $(f)$. 

\sn
So we get
\begin{enumerate}
    \item [$(*)_3$] There is a $(J,\mu^+)$-free $\Lambda \subseteq \prod\limits_{i<\kappa} d_i$ of cardinality $\mu_\bullet$. 
\end{enumerate}
[Why? Recalling $\kappa_\bullet \defeq \cf(\mu_\bullet) < \mu < \mu_\bullet$, there exists an increasing sequence of cardinals 
$\LL \lambda_\eps : \eps <\kappa_\bullet\RR$ with $\lambda_0 > \mu$ and limit $\mu_\bullet$.

Applying the conclusion of \ref{d17}(1)+(2) which we just obtained in {$(*)_2$}, for each $\eps < \kappa_\bullet$ there exists a $\mu$-free subset $\Lambda_\eps \subseteq {}^\kappa\mu$ of cardinality $\lambda_\eps$.
For each $\eta \in \Lambda_\eps$, define $\nu_\eta^\eps \defeq \LL \kappa_\bullet \cdot \eta(i) + \eps : i < \kappa\RR$.
Now $\Lambda = \{\nu_\eta^\eps : \eps < \kappa_\bullet,\ \eta \in \Lambda_\eps\}$ as promised.]

\begin{enumerate}
    \item [$(*)_4$] Let $\LL \eta_\alpha^* : \alpha < \mu_\bullet\RR$ list the members of $\Lambda$ without repetition.
\end{enumerate}

Third, applying Definition \ref{d3}(3) and clause (e) of the assumption, recalling \cite[Ch.II, 3.1]{Sh:g},
\begin{enumerate}
    \item [$(*)_5$] There is a $<_{J_\bullet}$-increasing sequence 
    $\bar f = \LL f_\gamma : \gamma < \lambda^+\RR \subseteq {}^{\kappa_\bullet} \mu_\bullet$ (hence all functions in the sequence are pairwise $\neq_{J_\bullet}$).
\end{enumerate}

\sn
By \cite{Sh:g}, 
\begin{enumerate}
    \item [$(*)_6$] Without loss of generality, $\bar f \rest \lambda = \LL f_\gamma : \gamma < \lambda\RR$ is a $J_\bullet$-free sequence and 
    $\gamma < \lambda \Rightarrow f_\gamma < f_\lambda$.
\sn
    \item [$(*)_7$] For each $i < \kappa$, we choose a one-to-one function $g_i$ such that
    \begin{enumerate}
        \item  $\dom(g_i) \defeq \bigcup\limits_{b \in J_\bullet^+} 
        {}^b(\alpha_i^\bullet)$
\sn
        \item  $\rang(g_i) \subseteq \mu$
    \end{enumerate}
\end{enumerate}
[Why does such a $g_i$ exist? 

The set $W_i \defeq \bigcup\limits_{b \in J_\bullet^+} 
{}^b(\alpha_i^\bullet)$ is a subset of $\bigcup\limits_{u \subseteq \kappa_\bullet} 
{}^u(\alpha_i^\bullet)$, {which} {has cardinality $|\alpha_i^\bullet|^{\kappa_\bullet}$;} this is $<\mu$ by {clause $(f)$ of our assumptions}. So clearly there is an injection from $W_i$ into $\mu$; {this satisfies {$(*)_7$}, and so it will be our $g_i$.]}

\medskip
\begin{enumerate}
    \item [$(*)_8$] For each $\gamma < \lambda$ we choose $\eta_\gamma \in {}^\kappa\mu$ as follows:

    For each $i < \kappa$ let $\nu_{\gamma,i} \defeq \big\LL \eta_{f_\gamma(\xi)}^*(i) : \xi < \kappa_\bullet \big\RR \in {}^{\kappa_\bullet}(\alpha_i^\bullet)$. 
    Now let
    $$
    \eta_\gamma \defeq  \big\LL g_i (\nu_{\gamma,i} ) : i < \kappa \big\RR.
    $$
\end{enumerate}
Lastly,
\begin{enumerate}
    \item [$(*)_9$] $\bar\eta = \LL\eta_\gamma : \gamma < \lambda\RR$ is as promised.
\end{enumerate}
Why? Clearly $\bar\eta$ is a sequence of members of ${}^\kappa\mu$, without repetition. Let $u \in [\lambda]^{\leq\mu}$, and we shall prove that $\bar\eta \rest u$ is free.

\begin{enumerate}
    \item [$(*)_{9.1}$] Without loss of generality $|u| < \mu$.
\end{enumerate}
[Why? By \cite{Sh:266}.]

\begin{enumerate}
    \item [$(*)_{9.2}$] We can find $\LL b_\gamma^* : \gamma \in u\RR \subseteq J_\bullet$ such that \underline{if} $\gamma \neq \eps \in u$ and $j \in \kappa_\bullet \setminus \big( b_\gamma^* \cup b_\eps^* \big)$ \underline{then} $f_\gamma(j) \neq f_\eps(j)$.
\end{enumerate}
[Why? By $(*)_6$.]

\begin{enumerate}
     \item [$(*)_{9.3}$] Let $W \defeq \big\{ f_\gamma(j) : \gamma \in u,\ j < \kappa_\bullet \big\}$ (so $W \in [\mu_\bullet]^{<\mu}$). Hence there exists a sequence $\LL a_\beta : \beta \in W\RR \in {}^W\!\!J$ witnessing that $\bar f \rest W$ is free.
\sn
     \item [$(*)_{9.4}$] We can choose $\LL a_\gamma : \gamma \in u\RR$ such that
     \begin{itemize}
         \item $a_\gamma \in J$
\sn
         \item $\{ j < \kappa_\bullet : a_{f_\gamma(j)} \subseteq a_\gamma\} \in J_\bullet^+$.
     \end{itemize}
\end{enumerate}
[Why? As $b_\gamma \in J_\bullet^+$, $b_\gamma^* \in J_\bullet$, and $J_\bullet$ is $\kappa^+$-complete.]

\begin{enumerate}
    \item [$(*)_{9.5}$] If $\gamma \neq \eps \in u$ and $j \in \kappa \setminus \big( a_\gamma \cup a_\eps \big)$ \underline{then} $\eta_\gamma(j) \neq \eta_\eps(j)$.
\end{enumerate}
[Why?  Put clauses $(*)_{9.1}$-$(*)_{9.4}$ together.] 

\sn
Now we are done.
\end{PROOF}

\bn
\begin{claim}\label{d26}
If $\mu$ is strong limit, $\kappa \defeq \cf(\mu) < \mu$, and $\lambda \in (\mu,2^\mu)$, \underline{then} there is a $\mu^+$-free $\lambda \subseteq {}^\kappa\mu$ of cardinality $\lambda$.
\end{claim}

\begin{PROOF}{\ref{d26}}
Let $\Theta_{\mu,\kappa} \defeq \big\{\chi \in (\mu,\lambda] : \cf(\chi) \in [\kappa^+,\mu) \text{ and } \pp_{\cf(\chi)\text{-complete}}(\chi) \geq \lambda^+ \big\}$.

If $\Theta_{\mu,\kappa}$ is empty then we can apply \ref{d17}. 

\sn
[Why? Choose $\kappa_* \defeq \kappa$, $\mu_* \defeq \mu$, $J \defeq J_\kappa^\bd$, and we have to verify the assumptions of \ref{d17}. Clauses $(a)$, $(b)$, and $(c)$ are obvious. Clause $(e)$ says `$\alpha < \mu \Rightarrow |\alpha|^\kappa < \mu$,' and this follows from $\mu$ being strong limit. 
Lastly, clause $(d)$ holds by \ref{d3}(4).]

So assume $\Theta_{\mu,\kappa} \neq \varnothing$. Let $\mu_\bullet \defeq \min\Theta_{\mu,\kappa}$, and apply \ref{d20}.

\sn
[Why can we do this? We should check assumptions \ref{d20}$(a)$-$(f)$. Choose $\chi \defeq \lambda$, $\kappa_\bullet \defeq \cf(\mu_\bullet)$, and $J_\bullet$ a $\kappa_\bullet$-complete ideal on $\kappa_\bullet$ satisfying $\pp_{J_\bullet}(\mu_\bullet) \geq \lambda^+$.

Clauses $(a)$ and $(b)$ are obvious. 

Clause $(c)$ says `$T_J^+(\mu) > \chi^+$', which holds because $J \defeq J_\kappa^\bd$ and $\mu$ is strong limit of cofinality $\kappa$.

In clause $(d)$, the first statement ($\cf(\mu_\bullet) \in \big[ \kappa^+,\mu \big)$) holds by the definition of $\Theta_{\mu,\kappa}$ and choice of $\mu_\bullet$. The second statement holds by the same reasoning.

Clause $(e)$ holds as $\mu_\bullet \in \Theta_{\mu,\kappa}$ and the choice of $J_\bullet$.

Clause $(f)$ holds as $\mu > \kappa_\bullet$ is strong limit.]
\end{PROOF}

\mn
\begin{discussion}\label{d29}
1) In \ref{d20}, there is no harm in adding ``$J_\bullet$ is $\kappa_\bullet$-complete."

\mn
2) Furthermore, if we add `$\kappa_\bullet > 2^\kappa$' then we can weaken clause \ref{d20}(c) to `$\bfT_J^+(\mu) > \chi^+$.'
\end{discussion}

\bibliographystyle{amsalpha}
\bibliography{shlhetal}
\end{document}